\documentclass[twoside,10pt,a4paper]{article}
\usepackage[colorlinks=true]{hyperref}                        
\hypersetup{ linkcolor=blue,  filecolor=black,  urlcolor=red,  citecolor=black}

\usepackage{amsmath,latexsym,amsfonts,indentfirst,amsthm,amsxtra,amssymb,bm,extarrows}
\usepackage{mathrsfs}
\usepackage[numbers,sort&compress]{natbib}
\usepackage[perpage,symbol*]{footmisc}
\usepackage{relsize}
\usepackage{xcolor}
\usepackage{BOONDOX-cal}
\usepackage{authblk}

\numberwithin{equation}{section}

\newtheorem{lemma}{Lemma}[section]
\newtheorem{proposition}[lemma]{Proposition}
\newtheorem{theorem}{Theorem}[section]

\theoremstyle{remark}

\allowdisplaybreaks[1]

\usepackage{accents}  
\makeatletter
\def\wideubar{\underaccent{{\cc@style\underline{\mskip10mu}}}}
\def\Wideubar{\underaccent{{\cc@style\underline{\mskip8mu}}}}
\makeatother
\makeatletter
\def\widebar{\accentset{{\cc@style\underline{\mskip10mu}}}}
\def\Widebar{\accentset{{\cc@style\underline{\mskip8mu}}}}
\makeatother

\newcommand{\lrn}{\left\vert\kern-0.3ex\left\vert\kern-0.3ex\left\vert}
\newcommand{\rrn}{\right\vert\kern-0.3ex\right\vert\kern-0.3ex\right\vert}

\begin{document}
	
	\title{
	\bf  Global existence and optimal decay rate of the classical solution to 3-D Radiative Hydrodynamics model with or without Heat Conductivity}

\author[a]{Guiqiong Gong}
\author[b,c]{Jiawei Zhou}
\author[b,c]{Boran Zhu\thanks{Corresponding author at:School of Mathematics and Statistics, Wuhan University, Wuhan 430072, China.\\
		Email: boran\_zhu@outlook.com(B. Zhu); gqgong@swjtu.edu.cn(G, Gong); zhoujiawei@whu.edu.cn(J. Zhou)}}
\affil[a]{\small School of Mathematics, Southwest Jiaotong University, Chengdu 611756, China }
\affil[b]{\small School of Mathematics and Statistics, Wuhan University, Wuhan 430072, China}
\affil[c]{\small Computational Science Hubei Key Laboratory, Wuhan University, Wuhan 430072, China}
\date{\empty}

\maketitle
	\begin{abstract}
        The classical solution of the 3-D radiative hydrodynamics model is studied in $H^k$-norm under two different conditions, with and without heat conductivity. We have proved the following results in both cases. First, when the $H^k$ norm of the initial perturbation around a constant state is sufficiently small and the integer $k\geq2$, a unique classical solution to such Cauchy problem is shown to exist. Second, if we further assume that the $L^1$ norm of the initial perturbation is small too, the i-order($0\leq i\leq k-2$) derivative of the solutions have the decay rate of $(1+t)^{-\frac 34-\frac i2}$ in $H^2$ norm. Third, from the results above we can see that for radiative hydrodynamics, the radiation can do the same job as the heat conduction, which means if the thermal conductivity coefficient turns to $0$, because of the effect of radiation, the solvability of the system and decay rate of the solution stay the same.  
	
		\vspace{2mm}
		

		\vspace{2mm}
	\end{abstract}
	

	\section{ Introduction}

   The key aim of studying the radiation hydrodynamics is to understand the motion of compressible fluids including radiation effects. The importance of thermal radiation in physical problems increases as the temperature is raised. At moderate temperatures, the role of the radiation is transporting energy by radiative processes. At higher temperatures, the energy and momentum densities of the radiation field may become dominate the corresponding fluid quantities. In this paper, we consider a radiative hydrodynamics model of the compressible Navier-Stokes equations coupled with an elliptic equation for radiative flux, which is used to describe the motion of viscous and heat-conducting fluids with radiative effects.
   This model always can be used to simulate, for instance, nonlinear stellar pulsation, supernova explosions and stellar winds astrophysics and so on. The compressible radiative hydrodynamics system takes the form of
	\begin{equation}
		\label{cns1}
		\left\{	
		\begin{aligned}
		&\rho_t+{\rm div}(\rho \bm{u})=0,\\
		&(\rho \bm{u})_t+{\rm div}(\rho \bm{u}\otimes\bm{u})
		+\nabla P= 2\mu{\rm div} \left(\mathcal{D}\left(\bm{u} \right)\right)+\mu'\nabla{\rm div}\bm{u},\\
		&(\rho E)_t+{\rm div}(\rho \bm{u}E+\bm{u}P)+{\rm div}\bm{q}
		=\kappa\Delta\theta+{\rm div}\left(2\mu \mathcal{D}\left(\bm{u} \right)\cdot\bm{u}+\mu'{\rm div}\bm{u}\ \bm{u}\right),\\
		&-\nabla{\rm div}\bm{q}+\bm{q}+\nabla\left(\theta^4\right)=0,
		\end{aligned}	
		\right.
	\end{equation}
    Here $\rho$, $\theta$, vectors $\bm{u},\ \bm{q}\in\mathbb{R}^3$ are respectively the mass density, absolute temperature, velocity field and radiative heat flux. Constants $\mu, \mu'>0$ denote coefficients of viscosity and $\kappa>0$ is heat conduction. For ideal polytropic gases, the specific total energy $E$, the pressure $P$ and the internal energy $e$ are prescribed as that
	\begin{equation}\label{1.2}
	E=e+\tfrac{1}{2}|\bm{u}|^2,\quad P=R\rho\theta,
	\quad e=C_v\theta,
	\end{equation}
    where $A>0$ and $R>0$ are the specific gas constants, $C_v$ denotes specific heat at constant volume.
	The fluid is considered to be Newtonian, so $\mathcal{D}\left(\bm{u}\right)$ takes the form as 
	\begin{equation}
	\label{VST}
	\mathcal{D}\left(\bm{u}\right)=\frac12\left(\nabla\bm{u}+(\nabla\bm{u})^{\mathsf{T}}\right),
	\end{equation}

	In this paper, we consider the system \eqref{cns1}--\eqref{1.2} in space $(x,t)\in \mathbb{R}^3\times\mathbb{R}^{+}$ with initial data that
	\begin{align} \label{initial0}
	(\rho,\bm{u},\theta)(0,\bm{x})
	=(\rho_0,\bm{u}_0,\theta_0)(x) \quad {\rm for}\  \ \bm{x}\in\mathbb{R}^3,
	\end{align}
    and the far field behavior that
    \begin{align} \label{initial1}
	(\rho,\bm{u},\theta)(t,\bm{x})
	\rightarrow(\bar{\rho},0,\bar{\theta}) \quad {\rm as}\  \ |\bm{x}|\rightarrow 0,t\geq 0,
	\end{align}
    without loss of generality, we set $\bar{\rho}=\bar{\theta}=1$ in the following paper.

     Due to the importance of radiative hydrodynamics in mathematics and physics, there is a lot of literature devoted to the mathematical theory of it. In one-dimensional case, the radiation have a quite good effect on the system. Even if we eliminate the heat conduction and viscosity, unlike the Cauchy problem of one-dimensional compressible Euler equation whose smooth solutions must blow up in finite time, the Cauchy problem of system \eqref{cns1} can still guarantee a unique global smooth solution while the initial date is sufficiently smooth and perturbation is small. The relevant study is abundant, one can refer to \citet{MR1679939,MR2022134}, \citet{MR2820208} for the global existence of classical solutions and \citet{MR4085490} for pointwise structure. About the stability of elementary waves, one can refer to \citet{MR2371479} and \citet{MR2360621} for the shock wave, \cite{MR2836843} for rarefaction wave and \cite{MR2812581} for viscous contact wave. As for the composite of several elementary waves, there are \citet{MR3906863} about two viscous shock waves, \citet{MR3015674} about rarefaction and contact waves, \citet{MR2873128} about viscous contact wave and rarefaction waves. With additional heat conduction and viscosity, \citet{MR3595322} has studied the large-time behavior toward the combination of two rarefaction waves and viscous contact wave. Under this circumstance, large initial perturbation problem can be considered. About the relevant work, we refer to \citet{MR20220425} and \citet{2204.04760}.
     In multi-dimensional case, see \cite{MR2329017} for the global existence of the weak entropy solution, \cite{MR2527538} for pointwise estimates of the global classical solutions, \cite{MR3887119} for the asymptotic stability of planar rarefaction wave, \cite{MR3989198,MR4085632} for radial symmetric classical solutions in an exterior domain and a bounded concentric annular domain, \cite{MR2402882,MR2414408} for decay rates of the planar rarefaction waves. Close to our topics, \citet{MR2815767} obtained the global existence of the classical solution and the decay rate by energy method.
       
     In this paper, we have three main results. First, we improve the work in \cite{MR2815767} in which, the global existence of the classical solution is obtained under the assumption that initial data is small enough in $H^k(k\geq4)$ and the thermal conductivity coefficient $\kappa\neq0$. We obtain the same global existence while $k\geq2$. Furthermore, the same results is proved for the case that $\kappa=0$. Second, with additional assumption that the $L^1$-norm of the initial data is sufficiently small, we obtain the decay rate of the solutions which is the same as that for the corresponding linear system. Such decay rate is known as the optimal decay rate for the nonlinear system. (With further assumptions on the initial data, we can also obtain the lower bound of the decay rate.) The method we use to derive the optimal decay rate improves that in \cite{MR3473451} and \cite{MR3498181}. The strategy in \cite{MR3473451} is not allow us to obtain the optimal decay rate of higher derivatives and that in \cite{MR3498181} although enable us to derive the optimal decay rate for higher older derivatives, we cannot get the optimal decay rate for $k-2$ order derivative, when the initial data belongs to $H^k(k\geq4)$. Because there is no decay in $k$ and $k-1$ order derivatives. We overcome such difficulties by applying inductive method and treat $H^2$ norm of the derivatives as a whole unit to obtain the decay rate of $k$ and $k-1$ order derivatives. The reason why optimal rate for $k-2$ order derivate is so important is that it is the best result we can get because of the application of Duhamel's principle. Third, we can see from the results above that for radiative hydrodynamics, the radiation can do the same job as the heat conduction, which means if the thermal conductivity coefficient turns to $0$, because of the effect of radiation, the solvability of the system and the decay rate of the solutions stay the same. We state our main results in the following Theorems.
\begin{theorem}\label{th1.1}
	If there exists a small enough constant $\epsilon_1>0$ such that the initial data satisfies 
	  \begin{equation*}
	    \|(\rho_0-1,\bm{u}_0,\theta_0-1)\|_{H^2}\leq\epsilon_1,
	  \end{equation*}
    the Cauchy problem \eqref{cns1}--\eqref{initial1} admits a unique global solution $\bm{V}=(\rho,\bm{u},\theta,\bm{q})$. When $\kappa\neq0$, the solution $\bm{V}$ satisfies
	  \begin{eqnarray*}
	    &\qquad\ \  \rho-1\in C\left([0,\infty);H^2\left(\mathbb{R}^3\right)\right)\cap C^1\left([0,\infty);H^1\left(\mathbb{R}^3\right)\right),\\[2mm]
	    &(\bm{u},\theta-1,\bm{q})\in C\left([0,\infty);H^2\left(\mathbb{R}^3\right)\right)\cap C^1\left([0,\infty);L^2\left(\mathbb{R}^3\right)\right),\\[2mm]
	    &\nabla\rho\in L^2\left([0,\infty);H^{1}\left(\mathbb{R}^3\right)\right),\  \nabla(\bm{u},\theta),\bm{q}\in L^2\left([0,\infty);H^2\left(\mathbb{R}^3\right)\right)
	  \end{eqnarray*}
  and
    \begin{eqnarray*}
      &\left\|(\rho-1,\bm{u},\theta-1,\bm{q})(t)\right\|_{H^2}^2
      +\int_{0}^{t}\|\nabla\rho(\tau)\|_{H^{1}}^2
        +\|\nabla(\bm{u},\theta)(\tau)\|_{H^2}^2+\|\bm{q}(\tau)\|_{H^2}^2{\rm d}\tau\\[2mm]
      &\leq C\ \|(\rho_0-1,\bm{u}_0,\theta_0-1)\|_{H^2}^2.
    \end{eqnarray*}
 When $\kappa=0$, the solution $\bm{V}$ satisfies
	  \begin{eqnarray*}
	  	&(\rho-1,\theta-1,\bm{q})\in C\left([0,\infty);H^2\left(\mathbb{R}^3\right)\right)\cap C^1\left([0,\infty);H^1\left(\mathbb{R}^3\right)\right),\\[2mm]
	    &\qquad\qquad\qquad\bm{u}\in C\left([0,\infty);H^2\left(\mathbb{R}^3\right)\right)\cap C^1\left([0,\infty);L^2\left(\mathbb{R}^3\right)\right),\\[2mm]
	    &\nabla(\rho, \theta)\in L^2\left([0,\infty);H^{1}\left(\mathbb{R}^3\right)\right),\  \nabla\bm{u},\bm{q}\in L^2\left([0,\infty);H^2\left(\mathbb{R}^3\right)\right)
      \end{eqnarray*}
 and
 \begin{eqnarray*}
 	&\left\|(\rho-1,\bm{u},\theta-1,\bm{q})(t)\right\|_{H^2}^2
 	+\int_{0}^{t}\|\nabla(\rho, \theta)(\tau)\|_{H^{1}}^2
 	+\|\nabla\bm{u}(\tau)\|_{H^2}^2+\|\bm{q}(\tau)\|_{H^2}^2{\rm d}\tau\\[2mm]
 	&\leq C\ \|(\rho_0-1,\bm{u}_0,\theta_0-1)\|_{H^2}^2.
 \end{eqnarray*}
\end{theorem}
\begin{theorem}\label{nablak1}
	If there exists a small enough constant $\epsilon_2>0$ such that the initial data satisfies  
	\begin{equation*}
	\|(\rho_0-1,\bm{u}_0,\theta_0-1)\|_{L^1\cap H^{k+2}}\leq\epsilon_2,\qquad(\ \text{for any integer}\ k\geq0) 
	\end{equation*}
	then there exists a positive constant $C$ such that
	\begin{align*}
		&\|\nabla^{i}(\rho-1,\bm{u}_0,\theta-1)(t)\|_2\leq C(1+t)^{-\frac 34-\frac i2}\|(\rho_0-1,\bm{u}_0,\theta_0-1)\|_{L^1\cap H^{i+2}}\quad(0\leq i\leq k),\\
		&\|\nabla^j\bm{q}\|\leq C\|\nabla^{j+1}\theta\|\quad(0\leq j\leq k+1),\qquad\text{for any $t\geq0$}.
	\end{align*}
\end{theorem}	
\begin{theorem}\label{updown0}
	 If there exists a small enough constant $\epsilon_3>0$ such that the initial data satisfies  
	 \begin{align*}
	 &\|(\rho_0-1,\bm{u}_0,\theta_0-1)\|_{L^1\cap H^{k}}\leq\epsilon_3,\quad(\ \text{for any integer}\ k\geq2)\\[1.5mm]
	 &C'\|\rho_0-1\|_{L^1}^2-c_0'\|(\bm{u}_0,\theta_0-1)\|_{L^1}^2>C\epsilon_3^2 
	 \end{align*}
	 where constants $C'$ and $c'$ shall be determined in Proposition \ref{Pro9}. Then there exist two positive constants $a'_1,a'_2$ such that the solutions $(\rho,\bm{u},\theta,\bm{q})$ satisfy
	\begin{equation*}
	a'_1(1+t)^{-\frac{3}{4}}\leq \|(\rho-1,\bm{u},\theta-1,{\rm div}\bm{q})(t)\|\leq a'_2(1+t)^{-\frac{3}{4}},\qquad \text{for any $t\geq0$}.
	\end{equation*}
\end{theorem}	

In this paper, due to the complexity of the model, we can not get an explicit expression of Green's function for the semigroup of linearized system, which makes it difficult to analyze the spectral and establish the optimal time decay rates. In order to overcome this difficult, according to \cite{MR2164944}, we have to analyze and decompose carefully the semigroup of the corresponding linearized system in order to deal with the system \eqref{cns1}.
To this end, we first analyze the property of $e^{t\bm{A}(\bm{\xi})}$ at both lower frequency and higher frequency respectively. And on that basis, then we analyze the asymptotical expansions of $\lambda_i$, $P_i (i=1,2,3,4)$  and $e^{t\bm{A}(\bm{\xi})}$ for both lower and higher frequencies, which plays a very important role in our later analysis.

The rest of this paper is arranged as follows. We obtain the global existence of the classical solutions to 3-D radiative hydrodynamics in section 2. In section 3, we shall perform the spectral analysis on the linearized system. Finally, the optimal decay rate of the nonlinear system will be derived in section 4.

\noindent \textbf{Notation.} For notational simplicity, throughout this paper, $L^q(\mathbb{R}^3)(1\leq q\leq \infty)$ stands for the usual Lebesgue space
on $\mathbb{R}^3$
with norm $\|\cdot\|_{L^q}$ and $H^k(\mathbb{R}^3)(k\in \mathbb{N})$ the usual Sobolev space in the $L^2$ sense with norm $\|\cdot\|_k$.
For simplicity, we introduce $\|\cdot\|=\|\cdot\|_{L^2(\mathbb{R}^3)}$, $\|(\ \cdot\ ,\ \cdot\ ,\ \cdot\ )\|_{L^q(\mathbb{R}^3)}=\|\cdot\|_{L^q(\mathbb{R}^3)}+\|\cdot\|_{L^q(\mathbb{R}^3)}+\|\cdot\|_{L^q(\mathbb{R}^3)}$ and $\lrn\cdot\rrn=\|\cdot\|_{L^{\infty}(\mathbb{R}^3)}$. The inner-product in $L^2(\mathbb{R}^3)$ is denoted by $\langle\ ,\ \rangle$. We denote by $C(I;H^p(\mathbb{R}^3))$ the space of continuous
functions on the interval $I$ with values in $H^p(\mathbb{R}^3)$ and $L^2(I;H^p(\mathbb{R}^3))$ the space of $L^2$-functions on $I$
with values in $H^p(\mathbb{R}^3)$. We introduce
  $A\lesssim B$ if $A\leq C B$ holds uniformly for some positive constant $C$ and similarly for $A\gtrsim B$.

\section{Global existence}	
\subsection{Reformulation of the problem}
We assume that the equilibrium state of the radiation hydrodynamic model \eqref{cns1} is trivial, taking the form of that
\begin{equation}
\bar{\rho}=1,\quad\bar{\bm{u}}=0,\quad\bar{\theta}=1,\quad\bar{\bm{q}}=0.
\end{equation}
Let $\mathcal{n}=\rho-1,\ \mathcal{m}=\theta-1.$ Then $\left[\mathcal{n},\bm{u},\mathcal{m},\bm{q}\right]$ satisfies that
\begin{equation}\label{cns2}
\left\{
\begin{aligned}
&\mathcal{n}_t+{\rm div}\bm{u}={\rm S_1},\\[4mm]
&\bm{u}_t+R\nabla\mathcal{n}+R\nabla\mathcal{m}-\mu\Delta\bm{u}-(\mu+\mu')\nabla{\rm div}\bm{u}={\rm S_2},\\[4mm]
&C_v\mathcal{m}_t+R{\rm div}\bm{u}+{\rm div}q-\kappa\Delta\mathcal{m}={\rm S_3},\\[4mm]
&-\frac{1}{4}\nabla{\rm div}\bm{q}+\frac 1 4\bm{q}+\nabla\mathcal{m}={\rm S_4},
\end{aligned}
\right.
\end{equation}
where
\begin{equation}
\left\{
\begin{aligned}
{\rm S_1}=&-\mathcal{n}{\rm div}\bm{u}-\bm{u}\cdot\nabla\mathcal{n},\\[4mm]
{\rm S_2}=&\ -\bm{u}\cdot\nabla\bm{u}-\frac{\mathcal{n}}{\mathcal{n}+1}\left[\mu\Delta\bm{u}+\left(\mu+\mu'\right)\nabla{\rm div}\bm{u}\right]+R\left(1-\frac{\mathcal{m}+1}{\mathcal{n}+1}\right)\nabla\mathcal{n},\\[3mm]
{\rm S_3}=&-C_v\bm{u}\cdot\nabla\mathcal{m}-R\mathcal{m}{\rm div}\bm{u}-\kappa\frac{\mathcal{n}}{\mathcal{n}+1}\Delta\mathcal{m}+\frac{2\mu}{\mathcal{n}+1}\mathcal{D}(\bm{u}):\mathcal{D}(\bm{u})+\\
&\ \frac{\mu'}{\mathcal{n}+1}\left({\rm div}\bm{u}\right)^2+\frac{\mathcal{n}}{\mathcal{n}+1}{{\rm div}\bm{q}},\\[3mm]
{\rm S_4}=&-\frac 3 2\nabla(\mathcal{m}^2)-\nabla(\mathcal{m}^3)-\frac 1 4\nabla(\mathcal{m}^4).
\end{aligned}
\right.
\end{equation}

The corresponding initial data is given by that
\begin{align} \label{initial2}
(\mathcal{n},\bm{u},\mathcal{m})(0,\bm{x})
=(\mathcal{n}_0,\bm{u}_0,\mathcal{m}_0)(\bm{x}) \quad {\rm for}\  \ \bm{x}\in\mathbb{R}^3.
\end{align}

We need to derive the a priori estimates of the solutions $(\mathcal{n},\bm{u},\mathcal{m},\bm{q})\in X(0,T;N)$ ($T>0$, $N>0$) to the initial boundary value problem \eqref{cns2}--\eqref{initial2}.
Here we define the sets
\begin{equation*}
\begin{split}
X_1(t_1,t_2 ;N):=\Big\{(\mathcal{n},\bm{u},\mathcal{m},\bm{q})\Big|\ &\mathcal{E}_1(t_1,t_2)\leq N^2;\ 
\mathcal{n}\in C\left([0,\infty);H^2\left(\mathbb{R}^3\right)\right)\cap C^1\left([0,\infty);H^1\left(\mathbb{R}^3\right)\right),\\[2mm]
&(\bm{u},\mathcal{m},\bm{q})\in C\left([0,\infty);H^2\left(\mathbb{R}^3\right)\right)\cap C^1\left([0,\infty);L^2\left(\mathbb{R}^3\right)\right),\\[2mm]
&\nabla\mathcal{n}\in L^2\left([0,\infty);H^{1}\left(\mathbb{R}^3\right)\right),\  \nabla(\bm{u},\mathcal{m}),\bm{q}\in L^2\left([0,\infty);H^2\left(\mathbb{R}^3\right)\right)\ \Big\},
\end{split}
\end{equation*}
\begin{equation*}
\begin{split}
X_2(t_1,t_2 ;N):=\Big\{(\mathcal{n},\bm{u},\mathcal{m},\bm{q})\Big|\ &\mathcal{E}_2(t_1,t_2)\leq N^2;\ 
\bm{u}\in C\left([0,\infty);H^2\left(\mathbb{R}^3\right)\right)\cap C^1\left([0,\infty);L^2\left(\mathbb{R}^3\right)\right),\\[2mm]
&(\mathcal{n},\mathcal{m},\bm{q})\in C\left([0,\infty);H^2\left(\mathbb{R}^3\right)\right)\cap C^1\left([0,\infty);L^2\left(\mathbb{R}^3\right)\right),\\[2mm]
&\nabla(\mathcal{n},\mathcal{m})\in L^2\left([0,\infty);H^{1}\left(\mathbb{R}^3\right)\right),\  \nabla\bm{u},\bm{q}\in L^2\left([0,\infty);H^2\left(\mathbb{R}^3\right)\right)\ \Big\},
\end{split}
\end{equation*}
for constants $N$, $t_1$, and $t_2$ ($t_1\leq t_2$),  where
\begin{align*} 
\mathcal{E}_1(t_1,t_2):=\sup_{t\in[t_1,t_2]}
\|(\mathcal{n},\bm{u},\mathcal{m},\bm{q})(t)\|_2^2
+\int_{t_1}^{t_2}\left(\|\nabla\mathcal{n}(\tau)\|_1^2
+\|\nabla(\bm{u},\mathcal{m})(\tau)\|_2^2+\|\bm{q}(\tau)\|^2_2\right)\mathrm{d}s,
\end{align*}
\begin{align*} 
\mathcal{E}_2(t_1,t_2):=\sup_{t\in[t_1,t_2]}
\|(\mathcal{n},\bm{u},\mathcal{m},\bm{q}(t)\|_2^2
+\int_{t_1}^{t_2}\left(\|\nabla(\mathcal{n}, \mathcal{m})(\tau)\|_1^2
+\|\nabla\bm{u}(\tau)\|_2^2+\|\bm{q}(\tau)\|^2_2\right)\mathrm{d}s.
\end{align*}
As usual, the global existence of solution to system \eqref{cns2} can be obtained by combining the local existence result with the a priori estimates.
\begin{proposition}[Local existence]\label{existence}
	If the initial data $(\mathcal{n}_0,\bm{u}_0,\mathcal{m}_0)\in H^2(\mathbb{R}^3)$ satisfies that
	  \begin{equation*}
	   \inf_{\bm{x}\in\mathbb{R}^3}\{\mathcal{n}_0(\bm{x})+1\}>0,
	  \end{equation*}
there exists a positive constant $T_0$ depending on $\mathcal{E}_i(0,0)$ such that the initial value problem \eqref{cns2} has a unique solution $(\mathcal{n},\bm{u},\mathcal{m},\bm{q})\in X_i(0,T_0;2N)$ which satisfies
	  \begin{equation*}
          \inf_{\bm{x}\in\mathbb{R}^3,0\leq t\leq T_0}\{\mathcal{n}(t,\bm{x})+1\}>0.
      \end{equation*}
$(i=1\  \text{as} \ \kappa\neq0;\ i=2 \ \text{as}\ \kappa=0)$
\end{proposition}
\begin{proposition}[A priori estimate]\label{a priori}
	Let the initial data $(\mathcal{n}_0,\bm{u}_0,\mathcal{m}_0)\in H^2(\mathbb{R}^3)$. Suppose that the initial value problem \eqref{cns2} has a solution $(\mathcal{n},\bm{u},\mathcal{m},\bm{q})\in X_i(0,T;\epsilon)$, where $T$ and $\epsilon$ are some positive constants and $\epsilon$ is small enough. Then there exists a positive constant $C$, which is independent of $T$ such that 
	\begin{equation*}
		\mathcal{E}_i(0,T)\leq C\left\|(\mathcal{n}_0,\bm{u}_0,\mathcal{m}_0)\right\|_2,
	\end{equation*}
	where $i=1\  \text{as} \ \kappa\neq0;\ i=2 \ \text{as}\ \kappa=0$.
\end{proposition}
Theorem \ref{th1.1} follows from Proposition \ref{existence} and \ref{a priori} by the standard continuity argument. The proof of Proposition \ref{existence} is standard, so we omit it here and one can refer to works \cite{MR2022134,2007On,1984Systems} for the details. Proposition \ref{a priori} shall be proved in Subsection \ref{proof of a priori}.
\subsection{Proof of a priori estimates}\label{proof of a priori}
In this subsection, we are going to present the proof of proposition \ref{a priori}. For later use, we list some Sobolev inequalities and Strauss Lemma as follows
\begin{lemma}[cf.\cite{MR0450957}]\label{inequality}
	Let $f\in H^2(\mathbb{R}^3)$. Then
	\begin{itemize}
		\item[$(i)$]$\|f\|_{L^{\infty}}\leq C\|\nabla f\|_1$; 
		\item[$(ii)$]$\|f\|_{L^6}\leq C\|\nabla f\|$;
		\item[$(iii)$]$\|f\|_{L^q}\leq C\|f\|_1$, $2\leq q\leq6$. 
	\end{itemize}
\end{lemma}
\begin{lemma}[cf.\cite{MR0233062}, Lemma 3.7]\label{Strauss}
	Let $M(t)$ be a non-negative continuous function of t satisfying the inequality
	\begin{equation*}
		M(t)\leq b_1+b_2M^{\gamma}(t),
	\end{equation*}
	in some interval containing $0$, where $b_1$ and $b_1$ are positive constants and $\gamma>1$. If $M(0)\leq b_1$ and
	\begin{equation*}
		b_1b_2^{\frac 1{\gamma-1}}<\left(1-\frac 1\gamma\right)\gamma^{-\frac 1{\gamma-1}},
	\end{equation*}
	then in the same interval 
	\begin{equation*}
		M(t)<\frac{b_1}{1-\frac{1}{\gamma}}\ .
	\end{equation*}
\end{lemma}
In what follows, a series of lemmas on energy estimates are given. Firstly, the basic energy estimate is obtained in the following lemma.
\begin{lemma}[basic energy estimate]\label{basiclemma}
	Under the assumption of proposition \ref{a priori}, there exist a sufficiently small constant $\delta\ (0<\epsilon_1\ll\delta\ll1)$ and a positive constant $C$, which are independent of $t$ such that
	\begin{equation}\label{basic}
	  \frac{\rm d}{{\rm d}t}\left(\|(\mathcal{n},\bm{u},\mathcal{m})\|^2+\delta\left\langle\bm{u},\nabla\mathcal{n}\right\rangle\right)+\|(\nabla\bm{u},\nabla\mathcal{n},\bm{q},{\rm div}\bm{q})\|^2+\kappa\|\nabla\mathcal{m}\|^2\leq C\delta\|(\nabla\mathcal{m},\nabla^2\bm{u})\|^2.
	\end{equation}
	\begin{proof}
		we compute that $\eqref{cns2}_1R\mathcal{n}+\eqref{cns2}_2\cdot\bm{u}+\eqref{cns2}_3\mathcal{m}+\eqref{cns2}_4\cdot\bm{q}$ and integrate the resultant equality over $\mathbb{R}^3$ to get
		\begin{multline}\label{basic1}
			\frac{1}{2}\frac{\rm d}{{\rm d}t}\|(\sqrt{R}\mathcal{n},\bm{u},\sqrt{C_v}\mathcal{m})\|^2
			+\mu\|\nabla\bm{u}\|^2+(\mu+\mu')\|{\rm div}\bm{u}\|^2+\kappa\|\nabla\mathcal{m}\|^2
			+\frac 1 4\|\bm{q}\|^2+\frac 1 4\|{\rm div}\bm{q}\|^2
			\\\leq\int_{\mathbb{R}^3}{\rm S_1}R\mathcal{n} \mathrm{d}\bm{x}+\int_{\mathbb{R}^3}{\rm S_2}\cdot\bm{u} \mathrm{d}\bm{x}+\int_{\mathbb{R}^3}{\rm S_3}\mathcal{m} \mathrm{d}\bm{x}+\int_{\mathbb{R}^3}{\rm S_4}\cdot\bm{q} \mathrm{d}\bm{x}.
		\end{multline}
	The four terms on the right side can be estimated as follows. For the first term, we have it from Lemma \ref{inequality} that
	\begin{align}\notag
		\int_{\mathbb{R}^3}{\rm S_1}R\mathcal{n} \mathrm{d}\bm{x}
		&\lesssim\|\mathcal{n}\|_{L^3}\|\mathcal{n}\|_{L^6}\|{\rm div}\bm{u}\|+\|\mathcal{n}\|_{L^3}\|\bm{u}\|_{L^6}\|\nabla\mathcal{n}\|\\\notag
		&\lesssim \|\mathcal{n}\|_1\|\nabla\mathcal{n}\|\|{\rm div}\bm{u}\|+\|\mathcal{n}\|_1\|\nabla\bm{u}\|\|\nabla\mathcal{n}\|\\[2mm]\notag
		&\lesssim \|\mathcal{n}\|_1\left(\|\nabla\mathcal{n}\|^2+\|{\rm div}\bm{u}\|^2\right)+\|\mathcal{n}\|_1\left(\|\nabla\bm{u}\|^2+\|\nabla\mathcal{n}\|^2\right)\\[2mm]
		&\lesssim \epsilon\left(\|\nabla\bm{u}\|^2+\|\nabla\mathcal{n}\|^2\right).
	\end{align}	
	The other three terms can be estimated similarly.
	\begin{align}\notag
		\int_{\mathbb{R}^3}{\rm S_2}\cdot\bm{u} \mathrm{d}\bm{x}
		\lesssim&\|\bm{u}\|_1\|\nabla\bm{u}\|^2+\|\mathcal{n}\|_{L^{\infty}}\|\nabla\bm{u}\|^2+\|\bm{u}\|_{L^{\infty}}\|\nabla\mathcal{n}\|\|\nabla\bm{u}\|+\|\mathcal{n}\|_{L^{\infty}}\|{\rm div}\bm{u}\|^2\\\notag
		&+\|\bm{u}\|_{L^{\infty}}\|\nabla\mathcal{n}\|\|{\rm div}\bm{u}\|+\|(\mathcal{n},\mathcal{m})\|_1\|\nabla\mathcal{n}\|\|\nabla\bm{u}\|\\[2mm]\label{S11}
		\lesssim&\epsilon_1\left(\|\nabla\bm{u}\|^2+\|\nabla\mathcal{n}\|^2+\|{\rm div}\bm{u}\|^2\right),\\[3mm]\notag
		\int_{\mathbb{R}^3}{\rm S_3}\mathcal{m} \mathrm{d}\bm{x}
		\lesssim&\|\bm{u}\|_1\|\nabla\mathcal{m}\|^2+\|\mathcal{m}\|_1\|\nabla\mathcal{m}\|\|{\rm div}\bm{u}\|+\|\mathcal{n}\|_{L^{\infty}}\|\nabla\mathcal{m}\|^2+\|\mathcal{m}\|_{L^{\infty}}\|\nabla\mathcal{n}\|\|\nabla\mathcal{m}\|\\\notag
		&+\|\mathcal{m}\|_{L^{\infty}}\|\nabla\bm{u}\|^2+\|\mathcal{m}\|_{L^{\infty}}\|{\rm div}\bm{u}\|^2+\|\mathcal{n}\|_1\|\nabla\mathcal{m}\|\|{\rm div}\bm{q}\|\\[2mm]\label{S12}
		\lesssim&\epsilon_1\left(\|\nabla\mathcal{m}\|^2+\|{\rm div}\bm{u}\|^2+\|\nabla\mathcal{n}\|^2+\|\nabla\bm{u}\|^2+\|{\rm div}\bm{q}\|^2\right),\\[3mm]	\label{S13}
		\int_{\mathbb{R}^3}{\rm S_4}\cdot\bm{q} \mathrm{d}\bm{x}
		\lesssim&\left(\|\mathcal{m}\|_{L^{\infty}}+\|\mathcal{m}\|_{L^{\infty}}^2+\|\mathcal{m}\|_{L^{\infty}}^3\right)\|\nabla\mathcal{m}\|\|\bm{q}\|\lesssim\epsilon_1\left(\|\nabla\mathcal{m}\|^2+\|\bm{q}\|^2\right).
	\end{align}
	Substituting \eqref{S11}, \eqref{S12} and \eqref{S13} into \eqref{basic1}, we obtain 
	\begin{equation}\label{basic'}
		\frac{\rm d}{{\rm d}t}\|(\mathcal{n},\bm{u},\mathcal{m})\|^2+\|(\nabla\bm{u},\bm{q},{\rm div}\bm{q})\|^2+\kappa\|\nabla\mathcal{m}\|^2\leq C\epsilon_1\|\nabla(\mathcal{m},\mathcal{n})\|.
	\end{equation}
    Finally, we shall estimate $\|\nabla\mathcal{n}\|^2$. We multiply $\eqref{cns2}_2$ by $\nabla\mathcal{n}$ to get
\begin{align}\notag
\left\langle R\frac{\mathcal{m}+1}{\mathcal{n}+1}\nabla\mathcal{n},\nabla\mathcal{n}\right\rangle
+&\frac{\rm d}{{\rm d}t}\left\langle\bm{u},\nabla\mathcal{n}\right\rangle\\\notag
=&-\left\langle{\rm div}\bm{u},\mathcal{n}_t\right\rangle
-R\left\langle\nabla\mathcal{m},\nabla\mathcal{n}\right\rangle
-\left\langle\bm{u}\cdot\nabla\bm{u},\nabla\mathcal{n}\right\rangle
+\mu\left\langle\frac{1}{\mathcal{n}+1}\Delta\bm{u},\nabla\mathcal{n}\right\rangle\\\notag
&+\left(\mu+\mu'\right)\left\langle\frac{1}{\mathcal{n}+1}\nabla{\rm div}\bm{u},\nabla\mathcal{n}\right\rangle,
\end{align}
which combined with $\eqref{cns2}_1$ leads to
\begin{align}\notag
\|\nabla\mathcal{n}\|^2
+\frac{\rm d}{{\rm d}t}\left\langle\bm{u},\nabla\mathcal{n}\right\rangle
\lesssim&\lrn\mathcal{n}+1\rrn\|{\rm div}\bm{u}\|^2
+\lrn\bm{u}\rrn\|\nabla\mathcal{n}\|\|{\rm div}\bm{u}\|
+\|\nabla\mathcal{m}\|\|\nabla\mathcal{n}\|
+\lrn\bm{u}\rrn\|\|\nabla\bm{u}\|\nabla\mathcal{n}\|\\[2mm]\notag	
&+\|\Delta\bm{u}\|\|\nabla\mathcal{n}\|
+\|\nabla{\rm div}\bm{u}\|\|\nabla\mathcal{n}\|\\[2mm]\label{n1}
\lesssim&\epsilon_1\left(\|\nabla\bm{u}\|^2+\|\nabla\mathcal{n}\|^2\right)+\|\nabla\mathcal{m}\|^2+\|\nabla^2\bm{u}\|^2.
\end{align}
We shall multiply \eqref{n1} by $\delta$, add the resultant inequality to \eqref{basic'} and the proof is completed.
	\end{proof}
\end{lemma}
Now we are going to drive the estimates for higher order derivatives separately in two different cases, which are $\kappa\neq0$ and $\kappa=0$.

\begin{lemma}[$\kappa\neq0$]\label{kno0}
	Under the assumption of proposition \ref{a priori}, there exist a sufficiently small constant $\delta\ (0<\epsilon_1\ll\delta\ll1)$ and a positive constant $C$, which are independent of $t$ such that 
	\begin{equation}\label{apriori1}
		\frac{\rm d}{{\rm d}t}\left(\|\nabla(\mathcal{n},\bm{u},\mathcal{m})\|^2_1
		+\delta\sum_{|\alpha|=1}\left\langle\partial_{\bm{x}}^{\alpha}\bm{u},\nabla\partial_{\bm{x}}^{\alpha}\mathcal{n}\right\rangle\right)
		+\|\nabla(\mathcal{n},\bm{q})\|^2_1+\|\nabla^2(\bm{u},\mathcal{m})\|^2_1
		\leq C\epsilon_1\|\nabla\bm{u}\|^2.
	\end{equation}
	\begin{proof}
		For each multi-index $\alpha$ with $1\leq|\alpha|\leq2$, we apply $\partial^\alpha_{\bm{x}}$ to $\eqref{cns2}_1-\eqref{cns2}_4$, multiply them by $R\partial_{\bm{x}}^\alpha\mathcal{n}$, $\partial_{\bm{x}}^\alpha\bm{u}$, $\partial_{\bm{x}}^\alpha\mathcal{m}$, $\partial_{\bm{x}}^\alpha\bm{q}$ respectively, add the resultant equalities together and integrate them over $\mathbb{R}^3$ to get that
		  \begin{eqnarray}\notag
		  	\frac{\rm d}{{\rm d}t}&&\|\partial_{\bm{x}}^\alpha(\sqrt{R}\mathcal{n},\bm{u},\sqrt{C_v}\mathcal{m})\|^2+\mu\|\nabla\partial_{\bm{x}}^\alpha\bm{u}\|^2+(\mu+\mu')\|{\rm div}\partial_{\bm{x}}^\alpha\bm{u}\|^2+\kappa\|\nabla\partial_{\bm{x}}^\alpha\mathcal{m}\|^2+\frac{1}{4}\|\partial_{\bm{x}}^\alpha\bm{q}\|^2+\frac1 4\|{\rm div}\partial_{\bm{x}}^\alpha\bm{q}\|^2\\\notag
		  	&&=\int_{\mathbb{R}^3}\partial_{\bm{x}}^\alpha{\rm S_1}\ R\partial_{\bm{x}}^\alpha\mathcal{n} \mathrm{d}\bm{x}+\int_{\mathbb{R}^3}\partial_{\bm{x}}^\alpha{\rm S_2}\cdot\partial_{\bm{x}}^\alpha\bm{u} \mathrm{d}\bm{x}+\int_{\mathbb{R}^3}\partial_{\bm{x}}^\alpha{\rm S_3}\ \partial_{\bm{x}}^\alpha\mathcal{m} \mathrm{d}\bm{x}+\int_{\mathbb{R}^3}\partial_{\bm{x}}^\alpha{\rm S_4}\cdot\partial_{\bm{x}}^\alpha\bm{q} \mathrm{d}\bm{x}\\\notag
		  	&&=\int_{\mathbb{R}^3}\partial_{\bm{x}}^\alpha{\rm S_1}\ R\partial_{\bm{x}}^\alpha\mathcal{n} \mathrm{d}\bm{x}
		  	-\int_{\mathbb{R}^3}\partial_{\bm{x}}^{\alpha-1}{\rm S_2}\cdot\partial_{\bm{x}}^{\alpha+1}\bm{u} \mathrm{d}\bm{x}
		  	+\int_{\mathbb{R}^3}\partial_{\bm{x}}^{\alpha}{\rm S_3}\ \partial_{\bm{x}}^{\alpha}\mathcal{m} \mathrm{d}\bm{x}+\int_{\mathbb{R}^3}\partial_{\bm{x}}^\alpha{\rm S_4}\cdot\partial_{\bm{x}}^\alpha\bm{q} \mathrm{d}\bm{x}\\[2mm]\label{zong}
		  	&&:=\mathbb{I}_1+\mathbb{I}_2+\mathbb{I}_3+\mathbb{I}_4.
		  \end{eqnarray}
 Then we estimate $\mathbb{I}_i\ (i=1,2,3,4)$ respectively. The first one is $\mathbb{I}_1$.
		\begin{align}\notag
			\mathbb{I}_1\lesssim&\int_{\mathbb{R}^3}\left|\mathcal{n}\ {\rm div}\partial_{\bm{x}}^\alpha\bm{u}\ \partial_{\bm{x}}^\alpha\mathcal{n}\right|
			+\left|{\rm div}\bm{u}\partial_{\bm{x}}^\alpha\mathcal{n}\partial_{\bm{x}}^\alpha\mathcal{n}\right|
			+\sum_{|\beta|=1(|\alpha|=2)}\left|{\rm div}\partial_{\bm{x}}^\beta\bm{u}\ \partial_{\bm{x}}^{\alpha-\beta}\mathcal{n}\ \partial_{\bm{x}}^\alpha\mathcal{n}\right|	\\\notag
			&\qquad+\frac{1}{2}\bm{u}\nabla\left[\left(\partial_{\bm{x}}^\alpha\mathcal{n}\right)^2\right]
			+\left|\partial_{\bm{x}}^\alpha\bm{u}\nabla\mathcal{n}\partial_{\bm{x}}^\alpha\mathcal{n}\right|
			+\sum_{|\beta|=1(|\alpha|=2)}\left|\partial_{\bm{x}}^\beta\bm{u}\nabla\partial_{\bm{x}}^{\alpha-\beta}\mathcal{n}\partial_{\bm{x}}^\alpha\mathcal{n}\right|\ {\rm d}\bm{x}\\\notag
			\lesssim&\lrn \mathcal{n} \rrn\left\|{\rm div}\partial_{\bm{x}}^\alpha\bm{u}\right\|\left\|\partial_{\bm{x}}^\alpha\mathcal{n}\right\|
			+\lrn{\rm div}\bm{u}\rrn\|\partial_{\bm{x}}^\alpha\mathcal{n}\|^2
			+\sum_{|\beta|=1(|\alpha|=2)}\left\|\partial_{\bm{x}}^{\alpha-\beta}\mathcal{n}\right\|_{L^3}\left\|{\rm div}\partial_{\bm{x}}^\beta\bm{u}\right\|_{L^6}\|\partial_{\bm{x}}^\alpha\mathcal{n}\|\\\notag
			&+\lrn{\rm div}\bm{u}\rrn\|\partial_{\bm{x}}^\alpha\mathcal{n}\|^2
			+\|\nabla\mathcal{n}\|_{L^3}\|\partial_{\bm{x}}^\alpha\bm{u}\|_{L^6}\|\partial_{\bm{x}}^\alpha\mathcal{n}\|^2
			+\sum_{|\beta|=1(|\alpha|=2)}\lrn\partial_{\bm{x}}^\beta\bm{u}\rrn\left\|\nabla\partial_{\bm{x}}^{\alpha-\beta}\mathcal{n}\right\|\|\partial_{\bm{x}}^\alpha\mathcal{n}\|\\\notag
			\lesssim&\|\nabla\mathcal{n}\|_1\|{\rm div}\partial_{\bm{x}}^\alpha\bm{u}\|\|\partial_{\bm{x}}^\alpha\mathcal{n}\|
			+\|\partial_{\bm{x}}^\alpha\mathcal{n}\|\|\nabla{\rm div}\bm{u}\|_1\|\partial_{\bm{x}}^\alpha\mathcal{n}\|
			+\|\nabla\mathcal{n}\|_1\|\nabla^{2}{\rm div}\bm{u}\|\|\partial_{\bm{x}}^\alpha\mathcal{n}\|(|\alpha|=2)\\[2.5mm]\notag
			&+\|\partial_{\bm{x}}^\alpha\mathcal{n}\|\|\nabla{\rm div}\bm{u}\|_1\|\partial_{\bm{x}}^\alpha\mathcal{n}\|
			+\|\nabla\mathcal{n}\|_1\|\nabla\partial_{\bm{x}}^\alpha\bm{u}\|\|\partial_{\bm{x}}^\alpha\mathcal{n}\|
			+\|\nabla^2\mathcal{n}\|\|\nabla^{2}\bm{u}\|_1\|\partial_{\bm{x}}^\alpha\mathcal{n}\|(|\alpha|=2)\\[2.5mm]\label{i1}
			\lesssim&\epsilon_1\left(\|\nabla^{|\alpha|+1}\bm{u}\|^2+\|\nabla^{|\alpha|}\mathcal{n}\|^2+\|\nabla^{2}\bm{u}\|_{|1}^2\right).
			\end{align}
			Now we derive the estimate for $\mathbb{I}_2$.
			\begin{align}\notag
			\mathbb{I}_2\lesssim&\int_{\mathbb{R}^3}\left|\bm{u}\nabla\partial_{\bm{x}}^{\alpha-1}\bm{u}\partial_{\bm{x}}^{\alpha+1}\bm{u}\right|
			+\sum_{|\beta|=1(|\alpha|=2)}\left|\partial_{\bm{x}}^{\beta}\bm{u}\nabla\partial_{\bm{x}}^{\alpha-1-\beta}\bm{u}\partial_{\bm{x}}^{\alpha+1}\bm{u}\right|
			+\left|\mathcal{n}\Delta\partial_{\bm{x}}^{\alpha-1}\bm{u}\partial_{\bm{x}}^{\alpha+1}\bm{u}\right|\\\notag
			&\qquad+\sum_{|\beta|=1(|\alpha|=2)}\left|\partial_{\bm{x}}^\beta\mathcal{n}\Delta\partial_{\bm{x}}^{\alpha-1-\beta}\bm{u}\partial_{\bm{x}}^{\alpha+1}\bm{u}\right|
			+\left|\mathcal{n}\nabla{\rm div}\partial_{\bm{x}}^{\alpha-1}\bm{u}\partial_{\bm{x}}^{\alpha+1}\bm{u}\right|\\\notag
			&\qquad+\sum_{|\beta|=1(|\alpha|=2)}\left|\partial_{\bm{x}}^\beta\mathcal{n}\nabla{\rm div}\partial_{\bm{x}}^{\alpha-1-\beta}\bm{u}\partial_{\bm{x}}^\alpha\bm{u}\right|
			+\left|(\mathcal{m},\mathcal{n})\right|\left|\nabla\partial_{\bm{x}}^{\alpha-1}\mathcal{n}\partial_{\bm{x}}^{\alpha+1}\bm{u}\right|\\\notag
			&\qquad+\sum_{|\beta|=1(|\alpha|=2)}\left|\partial_{\bm{x}}^\beta(\mathcal{n},\mathcal{m})\nabla\partial_{\bm{x}}^{\alpha-1-\beta}\mathcal{n}\partial_{\bm{x}}^{\alpha+1}\bm{u}\right|\ {\rm d}\bm{x}\\\notag
			\lesssim&\lrn\bm{u}\rrn\|\nabla^{|\alpha|}\bm{u}\|\|\partial_{\bm{x}}^{\alpha+1}\bm{u}\|
			+\left\|\nabla\bm{u}\right\|_{L^3}\left\|\nabla\bm{u}\right\|_{L^6}\|\partial_{\bm{x}}^{\alpha+1}\bm{u}\|(|\alpha|=2)+
			\lrn\mathcal{n}\rrn\|\Delta\partial_{\bm{x}}^{\alpha-1}\bm{u}\|\|\partial_{\bm{x}}^{\alpha+1}\bm{u}\|\\[2.5mm]\notag
			&+\|\nabla\mathcal{n}\|_{L^3}\|\Delta\bm{u}\|_{L^6}\|\partial_{\bm{x}}^{\alpha+1}\bm{u}\|(|\alpha|=2)
			+\lrn\mathcal{n}\rrn\|\nabla{\rm div}\partial_{\bm{x}}^{\alpha-1}\bm{u}\|\|\partial_{\bm{x}}^{\alpha+1}\bm{u}\|\\[2.5mm]\notag
			&+\|\nabla\mathcal{n}\|_{L^3}\|\nabla{\rm div}\bm{u}\|_{L^6}\|\partial_{\bm{x}}^{\alpha+1}\bm{u}\|(|\alpha|=2)
			+\lrn(\mathcal{m},\mathcal{n})\rrn\|\nabla\partial_{\bm{x}}^{\alpha-1}\mathcal{n}\|\|\partial_{\bm{x}}^{\alpha+1}\bm{u}\|\\[2.5mm]\notag
			&+\|\nabla(\mathcal{m},\mathcal{n})\|_{L^3}\|\nabla\mathcal{n}\|_{L^6}\|\partial_{\bm{x}}^{\alpha+1}\bm{u}\|(|\alpha|=2)\\[2.5mm]\label{i2}
			\lesssim&\epsilon_1\left(\|\nabla^{|\alpha|+1}\bm{u}\|^2+\|\nabla^{|\alpha|}\bm{u}\|^2+\|\nabla^{|\alpha|}\mathcal{n}\|^2\right).
			\end{align}
			The third term $\mathbb{I}_3$ is estimated as follows,
\begin{eqnarray}\notag
\mathbb{I}_3\lesssim&&\int_{\mathbb{R}^3}\frac{1}{2}\bm{u}\nabla\left[\left(\partial_{\bm{x}}^\alpha\mathcal{m}\right)^2\right]
+\left|\partial_{\bm{x}}^\alpha\bm{u}\nabla\mathcal{m}\partial_{\bm{x}}^\alpha\mathcal{m}\right|
+\sum_{|\beta|=1(|\alpha|=2)}\left|\partial_{\bm{x}}^{\alpha-\beta}\bm{u}\nabla\partial_{\bm{x}}^{\beta}\mathcal{m}\partial_{\bm{x}}^{\alpha}\mathcal{m}\right|
+\left|\mathcal{m}{\rm div}\partial_{\bm{x}}^\alpha\bm{u}\partial_{\bm{x}}^\alpha\mathcal{m}\right|
\\\notag
&&\qquad+\left|\partial_{\bm{x}}^\alpha\mathcal{m}{\rm div}\bm{u}\partial_{\bm{x}}^\alpha\mathcal{m}\right|
+\sum_{|\beta|=1(|\alpha|=2)}\left|\partial_{\bm{x}}^{\alpha-\beta}\mathcal{m}{\rm div}\partial_{\bm{x}}^{\beta}\bm{u}\partial_{\bm{x}}^{\alpha}\mathcal{m}\right|	
+\left|\mathcal{n}\Delta\partial_{\bm{x}}^{\alpha-1}\mathcal{m}\partial_{\bm{x}}^{\alpha+1}\mathcal{m}\right|\\\notag
&&\qquad
+\sum_{|\beta|=1(|\alpha|=2)}\left|\partial_{\bm{x}}^\beta\mathcal{n}\Delta\partial_{\bm{x}}^{\alpha-1-\beta}\mathcal{m}\partial_{\bm{x}}^{\alpha+1}\mathcal{m}\right|
+\left|\partial_{\bm{x}}^{\alpha}\left[\mathcal{D}(\bm{u}):\mathcal{D}(\bm{u})\right]\partial_{\bm{x}}^{\alpha}\mathcal{m}\right|\\\notag
&&\qquad
+\left|\partial_{\bm{x}}^\alpha\mathcal{n}\left[\mathcal{D}(\bm{u}):\mathcal{D}(\bm{u})\right]\partial_{\bm{x}}^\alpha\mathcal{m}\right|
+\sum_{|\beta|=1(|\alpha|=2)}\left|\partial_{\bm{x}}^\beta\mathcal{n}\partial_{\bm{x}}^{\alpha-\beta}\left[\mathcal{D}(\bm{u}):\mathcal{D}(\bm{u})\right]\partial_{\bm{x}}^{\alpha}\mathcal{m}\right|
\\\notag	
&&\qquad
+\left|\partial_{\bm{x}}^{\alpha}\left[\left({\rm div}\bm{u}\right)^2\right]\partial_{\bm{x}}^{\alpha}\mathcal{m}\right|
+\left|\left({\rm div}\bm{u}\right)^2\partial_{\bm{x}}^{\alpha}\mathcal{n}\partial_{\bm{x}}^{\alpha}\mathcal{m}\right|
+\sum_{|\beta|=1(|\alpha|=2)}\left|\partial_{\bm{x}}^\beta\mathcal{n}\partial_{\bm{x}}^{\alpha-\beta}\left[\left({\rm div}\bm{u}\right)^2\right]\partial_{\bm{x}}^{\alpha}\mathcal{m}\right|
\\\notag
&&\qquad+\left|\partial_{\bm{x}}^{\alpha}\mathcal{n}{\rm div}\bm{q}\partial_{\bm{x}}^{\alpha}\mathcal{m}\right|
+\left|\mathcal{n}{\rm div}\partial_{\bm{x}}^{\alpha}\bm{q}\partial_{\bm{x}}^{\alpha}\mathcal{m}\right|
+\sum_{|\beta|=1(|\alpha|=2)}\left|\partial_{\bm{x}}^{\alpha-1-\beta}\mathcal{n}{\rm div}\partial_{\bm{x}}^{\beta}\bm{q}\partial_{\bm{x}}^{\alpha}\mathcal{m}\right|\\\notag
\lesssim&&\lrn{\rm div}\bm{u}\rrn\|\partial_{\bm{x}}^{\alpha}\mathcal{m}\|^2
+\|\nabla\mathcal{m}\|_{L^3}\|\partial_{\bm{x}}^{\alpha}\bm{u}\|_{L^6}\|\partial_{\bm{x}}^{\alpha}\mathcal{m}\|
+\lrn\nabla\bm{u}\rrn\|\nabla^2\mathcal{m}\|^2(|\alpha|=2)\\[2.5mm]\notag
&&+\lrn\mathcal{m}\rrn\|{\rm div}\partial_{\bm{x}}^{\alpha}\bm{u}\|\|\partial_{\bm{x}}^{\alpha}\mathcal{m}\|
+\|\partial_{\bm{x}}^{\alpha}\mathcal{m}\|\lrn{\rm div}\bm{u}\rrn\|\partial_{\bm{x}}^{\alpha}\mathcal{m}\|
+\|\nabla\mathcal{m}\|_{L^3}\|\nabla{\rm div}\bm{u}\|_{L^6}\|\partial_{\bm{x}}^{\alpha}\mathcal{m}\|(|\alpha|=2)\\[2.5mm]\notag
&&+\lrn\mathcal{n}\rrn\|\Delta\partial_{\bm{x}}^{\alpha-1}\mathcal{m}\|\|\partial_{\bm{x}}^{\alpha+1}\mathcal{m}\|
+\|\nabla
\mathcal{n}\|_{L^3}\|\Delta\mathcal{m}\|_{L^6}\|\partial_{\bm{x}}^{\alpha+1}\mathcal{m}\|(|\alpha|=2)\\[2.5mm]\notag
&&+\|\partial_{\bm{x}}^{\alpha}\mathcal{m}\|\lrn\nabla\bm{u}\rrn\|\nabla^{|\alpha|+1}\bm{u}\|
+\|\partial_{\bm{x}}^{\alpha}\mathcal{m}\|\|\nabla^2\bm{u}\|_{L^3}\|\nabla^2\bm{u}\|_{L^6}\|(|\alpha|=2)\\[2.5mm]\notag
&&+\|\partial_{\bm{x}}^{\alpha}\mathcal{n}\|\|\partial_{\bm{x}}^{\alpha}\mathcal{m}\|\lrn\left[\mathcal{D}(\bm{u}):\mathcal{D}(\bm{u})\right]\rrn
+\|\nabla\mathcal{n}\|_{L^6}\|\nabla\bm{u}\|_{L^6}\|\nabla^2\bm{u}\|_{L^6}\|\partial_{\bm{x}}^{\alpha}\mathcal{m}\|(|\alpha|=2)\\[2.5mm]\notag
&&+\|\partial_{\bm{x}}^{\alpha}\mathcal{m}\|\lrn{\rm div}\bm{u}\rrn\|\nabla^{|\alpha|}{\rm div}\bm{u}\|
+\|\partial_{\bm{x}}^{\alpha}\mathcal{m}\|\|\nabla{\rm div}\bm{u}\|_{L^3}\|\nabla{\rm div}\bm{u}\|_{L^6}\|(|\alpha|=2)\\[2.5mm]\notag
&&+\|\partial_{\bm{x}}^{\alpha}\mathcal{n}\|\|\partial_{\bm{x}}^{\alpha}\mathcal{m}\|\lrn\left({\rm div}\bm{u}\right)^2\rrn
+\|\nabla\mathcal{n}\|_{L^6}\|{\rm div}\bm{u}\|_{L^6}\|\nabla{\rm div}\bm{u}\|_{L^6}\|\partial_{\bm{x}}^{\alpha}\mathcal{m}\|(|\alpha|=2)\\[2.5mm]\notag
&&+\|\partial_{\bm{x}}^{\alpha}\mathcal{n}\|\lrn{\rm div}\bm{q}\rrn\|\partial_{\bm{x}}^{\alpha}\mathcal{m}\|
+\lrn\mathcal{n}\rrn\|{\rm div}\partial_{\bm{x}}^{\alpha}\bm{q}\|\|\partial_{\bm{x}}^{\alpha}\mathcal{m}\|
+\|\nabla\mathcal{n}\|_{L^3}\|\nabla{\rm div}\bm{q}\|_{L^6}\|\partial_{\bm{x}}^{\alpha}\mathcal{m}\|(|\alpha|=2)\\[2.5mm]\notag
\lesssim&&\epsilon_1\left(\|\nabla^2\bm{u}\|_1^2+\|\nabla^{|\alpha|}\mathcal{m}\|^2+\|\nabla^{|\alpha|+1}\bm{u}\|^2\right.\\\label{i3}
&&\left.\qquad\qquad\qquad\qquad\qquad\qquad\qquad+\|\nabla^{|\alpha|+1}\mathcal{m}\|^2(\kappa\neq0)+\|\nabla{\rm div}\bm{q}\|_1^2+\|\nabla^{|\alpha|}{\rm div}\bm{q}\|^2\right).	
\end{eqnarray}
Finally, we derive the estimate for $\mathbb{I}_4$ as follows,
\begin{eqnarray}\notag
	\mathbb{I}_4\lesssim&&\int_{\mathbb{R}^3}\partial_{\bm{x}}^{\alpha}\left(\frac 3 2\mathcal{m}^2+\mathcal{m}^3+\frac 1 4\mathcal{m}^4\right){\rm div}\partial_{\bm{x}}^{\alpha}\bm{q}\ {\rm d}\bm{x}\\[2.5mm]\notag
	\lesssim&&\epsilon_1\|\nabla^{|\alpha|}\mathcal{m}\|\|\nabla^{\alpha}{\rm div}\bm{q}\|
	+\epsilon_1\|\nabla\mathcal{m}\|_{L^3}\|\nabla\mathcal{m}\|_{L^6}\|\nabla^{|\alpha|}{\rm div}\bm{q}\|(|\alpha|=2)\\[2.5mm]\label{i4}
	\lesssim&&\epsilon_1\left(\|\nabla^{|\alpha|}\mathcal{m}\|^2+\|\nabla^{|\alpha|}{\rm div}\bm{q}\|^2\right).
\end{eqnarray}
Substitute \eqref{i1}, \eqref{i2}, \eqref{i3} and \eqref{i4} into \eqref{zong} and take the sum in terms of $\alpha$ to have
\begin{equation}\label{num}
	\frac{\rm d}{{\rm d}t}\|\nabla(\mathcal{n},\bm{u},\mathcal{m})\|_1^2
	+\|\nabla^2\left(\bm{u},\mathcal{m}\right)\|_1^2
	+\|\nabla\bm{q}\|_1^2
	+\|\nabla{\rm div}\bm{q}\|_1^2
	\lesssim\epsilon_1\left(\|\nabla\bm{u}\|^2+\|\nabla\mathcal{n}\|^2_1\right).
\end{equation}
To make up for the dissipation of $\mathcal{n}$, we apply $\partial_{\bm{x}}^\alpha$ to $\eqref{cns2}_2$, multiply it by $\nabla\partial_{\bm{x}}^\alpha\mathcal{n}$ and integrate the resultant equation over $\mathbb{R}^3$, where $|\alpha|=1$. Then we have
\begin{align}\notag
	&\left\langle R\frac{\mathcal{m}+1}{\mathcal{n}+1}\nabla\partial_{\bm{x}}^\alpha\mathcal{n},\nabla\partial_{\bm{x}}^\alpha\mathcal{n}\right\rangle
	+\frac{\rm d}{{\rm d}t}\left\langle\partial_{\bm{x}}^{\alpha}\bm{u},\nabla\partial_{\bm{x}}^{\alpha}\mathcal{n}\right\rangle\\\notag
	=&-\left\langle\partial_{\bm{x}}^{\alpha}{\rm div}\bm{u},\partial_{\bm{x}}^{\alpha}\mathcal{n}_t\right\rangle
	-\left\langle\nabla\mathcal{n}\partial_{\bm{x}}^\alpha\frac{\mathcal{m}+1}{\mathcal{n}+1},\nabla\partial_{\bm{x}}^\alpha\mathcal{n}\right\rangle
	-R\left\langle\nabla\partial_{\bm{x}}^{\alpha}\mathcal{m},\nabla\partial_{\bm{x}}^{\alpha}\mathcal{n}\right\rangle
	-\left\langle\partial_{\bm{x}}^{\alpha}\bm{u}\cdot\nabla\bm{u},\nabla\partial_{\bm{x}}^{\alpha}\mathcal{n}\right\rangle\\\notag
	&-\left\langle\bm{u}\cdot\nabla\partial_{\bm{x}}^{\alpha}\bm{u},\nabla\partial_{\bm{x}}^{\alpha}\mathcal{n}\right\rangle+\mu\left\langle\frac{1}{\mathcal{n}+1}\Delta\partial_{\bm{x}}^{\alpha}\bm{u},\nabla\partial_{\bm{x}}^{\alpha}\mathcal{n}\right\rangle
	+\mu\left\langle\Delta\bm{u}\partial_{\bm{x}}^{\alpha}\frac{1}{\mathcal{n}+1},\nabla\partial_{\bm{x}}^{\alpha}\mathcal{n}\right\rangle\\\notag
	&+\left(\mu+\mu'\right)\left\langle\frac{1}{\mathcal{n}+1}\nabla{\rm div}\partial_{\bm{x}}^{\alpha}\bm{u},\nabla\partial_{\bm{x}}^{\alpha}\mathcal{n}\right\rangle
	+\left(\mu+\mu'\right)\left\langle\nabla{\rm div}\bm{u}\partial_{\bm{x}}^{\alpha}\frac{1}{\mathcal{n}+1},\nabla\partial_{\bm{x}}^{\alpha}\mathcal{n}\right\rangle,
\end{align}
which combined with $\eqref{cns2}_1$ leads to
\begin{align}\notag
	\|\nabla\partial_{\bm{x}}^\alpha\mathcal{n}\|^2
	+\frac{\rm d}{{\rm d}t}&\left\langle\partial_{\bm{x}}^{\alpha}\bm{u},\nabla\partial_{\bm{x}}^{\alpha}\mathcal{n}\right\rangle\\[2mm]\notag
	\lesssim&\lrn\mathcal{n}+1\rrn\|\partial_{\bm{x}}^{\alpha}{\rm div}\bm{u}\|^2
	+\|{\rm div}\bm{u}\|_{L^3}\|\partial_{\bm{x}}^{\alpha}\mathcal{n}\|_{L^6}\|\partial_{\bm{x}}^{\alpha}{\rm div}\bm{u}\|
	+\lrn\bm{u}\rrn\|\nabla\partial_{\bm{x}}^{\alpha}\mathcal{n}\|\|\partial_{\bm{x}}^{\alpha}{\rm div}\bm{u}\|\\[2mm]\notag
	&+\|\nabla\mathcal{n}\|_{L^3}\|\partial_{\bm{x}}^{\alpha}\bm{u}\|_{L^6}\|\partial_{\bm{x}}^{\alpha}{\rm div}\bm{u}\|
	+\|\nabla\mathcal{n}\|_{L^3}\|\partial_{\bm{x}}^{\alpha}(\mathcal{m},\mathcal{n})\|_{L^6}\|\nabla\partial_{\bm{x}}^{\alpha}\mathcal{n}\|\\[2mm]\notag	
	&+\|\nabla\partial_{\bm{x}}^{\alpha}\mathcal{m}\|\|\nabla\partial_{\bm{x}}^{\alpha}\mathcal{n}\|
	+\|\nabla\bm{u}\|_{L^3}\|\partial_{\bm{x}}^{\alpha}\bm{u}\|_{L^6}\|\nabla\partial_{\bm{x}}^{\alpha}\mathcal{n}\|
	+\lrn\bm{u}\rrn\|\nabla\partial_{\bm{x}}^{\alpha}\bm{u}\|\|\nabla\partial_{\bm{x}}^{\alpha}\mathcal{n}\|\\[2mm]\notag
	&+\|\Delta\partial_{\bm{x}}^{\alpha}\bm{u}\|\|\nabla\partial_{\bm{x}}^{\alpha}\mathcal{n}\|
	+\|\partial_{\bm{x}}^{\alpha}\mathcal{n}\|_{L^3}\|\Delta\bm{u}\|_{L^6}\|\nabla\partial_{\bm{x}}^{\alpha}\mathcal{n}\|
	+\|\nabla{\rm div}\partial_{\bm{x}}^{\alpha}\bm{u}\|\|\nabla\partial_{\bm{x}}^{\alpha}\mathcal{n}\|\\[2mm]\notag
	&+\|\partial_{\bm{x}}^{\alpha}\mathcal{n}\|_{L^3}\|\nabla{\rm div}\bm{u}\|_{L^6}\|\nabla\partial_{\bm{x}}^{\alpha}\mathcal{n}\|.
\end{align}
That is
\begin{multline}\label{n}
    \|\nabla\partial_{\bm{x}}^\alpha\mathcal{n}\|^2
    +\frac{\rm d}{{\rm d}t}\left\langle\partial_{\bm{x}}^{\alpha}\bm{u},\nabla\partial_{\bm{x}}^{\alpha}\mathcal{n}\right\rangle\\
	\lesssim\epsilon_1\left(\|\nabla^{2}\bm{u}\|^2+\|\nabla^{2}\mathcal{m}\|^2+\|\nabla^3\bm{u}\|^2\right)
	+\left(\|\nabla^2\bm{u}\|^2+\|\nabla^{2}\mathcal{m}\|^2+\|\nabla^{3}\bm{u}\|^2\right).
\end{multline}
We take the sum of \eqref{n} in terms of $\alpha$, multiply it by $\delta$ and add the resultant inequality to \eqref{num}, which complete the proof.
  \end{proof}
\end{lemma}
We obtain the estimates for $\bm{q}$ and its higher order derivatives in the following lemma.
\begin{lemma}
	Under the assumption of proposition \ref{a priori}, there exists a positive constant $C$, which is independent of $t$ such that
	\begin{equation}
		\|\bm{q}(t)\|^2_2+\|{\rm div}\bm{q}(t)\|_2^2\lesssim\|\nabla\mathcal{m}(t)\|_1^2
	\end{equation}
\begin{proof}
First, multiply $\eqref{cns2}_4$ by $\bm{q}$ and integrate over $\mathbb{R}^3$ to get
\begin{equation}\label{q1}
	\|\bm{q}(t)\|^2+\|{\rm div}\bm{q}(t)\|^2\lesssim\|\nabla\mathcal{m}(t)\|^2
\end{equation}
Then, for each multi-index $\alpha$ with $1\leq|\alpha|\leq2$, we apply $\partial^\alpha_{\bm{x}}$ to $\eqref{cns2}_4$ and multiply it by $\partial_{\bm{x}}^\alpha\bm{q}$ and integrate the resultant equation over $\mathbb{R}^3$ and we obtain that
\begin{align}\notag
\qquad\|\partial_{\bm{x}}^{\alpha}\bm{q}\|^2
	+&\|{\rm div}\partial_{\bm{x}}^{\alpha}\bm{q}\|^2\\[3mm]\notag
	\lesssim&\int_{\mathbb{R}^3}\partial_{\bm{x}}^{\alpha}(\mathcal{m}+1)^4{\rm div}\partial_{\bm{x}}^{\alpha}\bm{q}\ {\rm d}\bm{x}\\\notag
	\lesssim&\lrn4(\mathcal{m}+1)^3\rrn\left\langle\partial_{\bm{x}}^{\alpha}\mathcal{m},{\rm div}\partial_{\bm{x}}^{\alpha}\bm{q}\right\rangle+\lrn12(\mathcal{m}+1)^2\rrn\left\langle\sum_{|\beta|=1(|\alpha|=2)}\partial_{\bm{x}}^{\beta}\mathcal{m}\partial_{\bm{x}}^{\alpha-\beta}\mathcal{m},{\rm div}\partial_{\bm{x}}^{\alpha}\bm{q}\right\rangle\\\notag
	\lesssim&\|\partial_{\bm{x}}^{\alpha}\mathcal{m}\|\|{\rm div}\partial_{\bm{x}}^{\alpha}\bm{q}\|+\|\nabla\mathcal{m}\|_{L^3}\|\nabla\mathcal{m}\|_{L^6}\|{\rm div}\partial_{\bm{x}}^{\alpha}\bm{q}\|(|\alpha|=2)\\[3.5mm]\label{q2}
	\lesssim&\epsilon_1\|{\rm div}\partial_{\bm{x}}^{\alpha}\bm{q}\|^2+\|\nabla^{|\alpha|}\mathcal{m}\|^2.
\end{align}
The lemma holds from \eqref{q2} combined with \eqref{q1}.
	\end{proof}
\end{lemma}
Next, we are going to drive the estimates for higher order derivatives when $\kappa=0$.
\begin{lemma}[$\kappa=0$]\label{kis0}
	Under the assumption of proposition \ref{a priori}, there exists two sufficiently small constants $\delta\ (0<\epsilon_1\ll\delta\ll1)$ and a positive constant $C$, which are all independent of $t$ such that 
	\begin{equation}\label{apriori2}
	\frac{\rm d}{{\rm d}t}\left(\|\nabla(\mathcal{n},\bm{u},\mathcal{m})\|^2_1
	+\delta\sum_{|\alpha|\leq1}\left\langle\partial_{\bm{x}}^{\alpha}\bm{u},\nabla\partial_{\bm{x}}^{\alpha}\mathcal{n}\right\rangle\right)
	+\|\nabla(\mathcal{n},\mathcal{m})\|^2_1+\|\nabla\bm{u}\|^2_2
	\leq C\epsilon_1\|\nabla\bm{u}\|^2.
	\end{equation}
	\begin{proof}
		For each multi-index $\alpha$ with $1\leq|\alpha|\leq2$, we apply $\partial^\alpha_{\bm{x}}$ to $\eqref{cns2}_1-\eqref{cns2}_3$, multiply them by $R\partial_{\bm{x}}^\alpha\mathcal{n}$, $\partial_{\bm{x}}^\alpha\bm{u}$, $\partial_{\bm{x}}^\alpha\mathcal{m}$ respectively, add the resultant equalities together and integrate them over $\mathbb{R}^3$ to get that
		\begin{eqnarray}\notag
		\frac{\rm d}{{\rm d}t}&&\|\partial_{\bm{x}}^\alpha(\sqrt{R}\mathcal{n},\bm{u},\sqrt{C_v}\mathcal{m})\|^2+\mu\|\nabla\partial_{\bm{x}}^\alpha\bm{u}\|^2+(\mu+\mu')\|{\rm div}\partial_{\bm{x}}^\alpha\bm{u}\|^2+\int_{\mathbb{R}^3}\partial_{\bm{x}}^\alpha{\rm div}\bm{q}\partial_{\bm{x}}^\alpha\mathcal{m}\ {\rm d}\bm{x}\\\notag
		&&=\int_{\mathbb{R}^3}\partial_{\bm{x}}^\alpha{\rm S_1}R\partial_{\bm{x}}^\alpha\mathcal{n} \ \mathrm{d}\bm{x}+\int_{\mathbb{R}^3}\partial_{\bm{x}}^\alpha{\rm S_2}\cdot\partial_{\bm{x}}^\alpha\bm{u} \ \mathrm{d}\bm{x}+\int_{\mathbb{R}^3}\partial_{\bm{x}}^\alpha{\rm S_3}\partial_{\bm{x}}^\alpha\mathcal{m} \ \mathrm{d}\bm{x}\\\notag
		&&=\int_{\mathbb{R}^3}\partial_{\bm{x}}^\alpha{\rm S_1}R\partial_{\bm{x}}^\alpha\mathcal{n} \ \mathrm{d}\bm{x}
		-\int_{\mathbb{R}^3}\partial_{\bm{x}}^{\alpha-1}{\rm S_2}\cdot\partial_{\bm{x}}^{\alpha+1}\bm{u} \ \mathrm{d}\bm{x}
		+\int_{\mathbb{R}^3}\partial_{\bm{x}}^{\alpha}{\rm S_3}\partial_{\bm{x}}^{\alpha}\mathcal{m} \ \mathrm{d}\bm{x}\\[2mm]
		&&:=\mathbb{I}_1+\mathbb{I}_2+\mathbb{I}_3.
		\end{eqnarray}
The estimates for $\mathbb{I}_i\ (i=1,2,3)$ we have in the proof of Lemma \ref{kno0}	still hold here. So we take the sum in terms of $\alpha$ to get
\begin{multline}\label{zong0}
\frac{\rm d}{{\rm d}t}\|\nabla(\mathcal{n},\bm{u},\mathcal{m})\|_1^2
+\|\nabla^2\bm{u}\|_1^2+\int_{\mathbb{R}^3}\nabla{\rm div}\bm{q}\cdot\nabla\mathcal{m}+\nabla^2{\rm div}\bm{q}\cdot\nabla^2\mathcal{m}\ {\rm d}\bm{x}\\
\lesssim\epsilon_1\left(\|\nabla\bm{u}\|^2
+\|\nabla\mathcal{n}\|^2_1
+\|\nabla\mathcal{m}\|^2_1
+\|\nabla{\rm div}\bm{q}\|_1^2\right)
\lesssim\epsilon_1\left(\|\nabla\bm{u}\|^2
+\|\nabla\mathcal{n}\|^2_1
+\|\nabla\mathcal{m}\|^2_1\right).
\end{multline}	
Here, we have used \eqref{q2} in the last inequality.		
Then, we multiply $\eqref{cns2}_4$ by $\nabla\mathcal{m}$ and integrate the resultant equation over $\mathbb{R}^3$ to get
\begin{align}\label{m1}
	\|\nabla\mathcal{m}\|^2-\int_{\mathbb{R}^3}\nabla{\rm div}\bm{q}\cdot\nabla\mathcal{m}\ {\rm d}\bm{x}\lesssim\|\bm{q}\|^2.
\end{align}
For each multi-index $\alpha$ with $|\alpha|=1$, we apply $\partial^\alpha_{\bm{x}}$ to $\eqref{cns2}_4$ and multiply it by $\partial_{\bm{x}}^\alpha\bm{q}$ and integrate the resultant equation over $\mathbb{R}^3$ to get
\begin{align*}
	\int_{\mathbb{R}^3}4(\mathcal{m}+1)^3\left(\nabla\partial_{\bm{x}}^\alpha\mathcal{m}\right)^2\ {\rm d}\bm{x}
	&-\int_{\mathbb{R}^3}\partial_{\bm{x}}^\alpha\nabla{\rm div}\bm{q}\cdot\nabla\partial_{\bm{x}}^\alpha\mathcal{m}\ {\rm d}\bm{x}\\\notag
	=&-\int_{\mathbb{R}^3}\partial_{\bm{x}}^\alpha\bm{q}\cdot\nabla\partial_{\bm{x}}^\alpha\mathcal{m}\ {\rm d}\bm{x}
	-\int_{\mathbb{R}^3}12(\mathcal{m}+1)^2\partial_{\bm{x}}^\alpha\mathcal{m}\nabla\mathcal{m}\cdot\nabla\partial_{\bm{x}}^\alpha\mathcal{m}\ {\rm d}\bm{x}\\
	=&-\int_{\mathbb{R}^3}{\rm div}\bm{q}\left(\partial_{\bm{x}}^\alpha\right)^2\mathcal{m}\ {\rm d}\bm{x}
	-\int_{\mathbb{R}^3}12(\mathcal{m}+1)^2\partial_{\bm{x}}^\alpha\mathcal{m}\nabla\mathcal{m}\cdot\nabla\partial_{\bm{x}}^\alpha\mathcal{m}\ {\rm d}\bm{x}\\[3mm]
	\lesssim&\quad\|{\rm div}\bm{q}\|\|\nabla^2\mathcal{m}\|+\|\nabla\mathcal{m}\|_{L^3}\|\nabla\mathcal{m}\|_{L^6}\left\|\nabla^2\mathcal{m}\right\|\\[3mm]
	\lesssim&\quad\epsilon_1\left\|\nabla^2\mathcal{m}\right\|^2+\|{\rm div}\bm{q}\|.
\end{align*}	
Take the sum of last inequality in terms of $\alpha$ to obtain that
\begin{align}\label{m2}
	\left\|\nabla^2\mathcal{m}\right\|^2-\int_{\mathbb{R}^3}\nabla^2{\rm div}\bm{q}\cdot\nabla^2\mathcal{m}\ {\rm d}\bm{x}\lesssim\|{\rm div}\bm{q}\|^2.
\end{align}	
With \eqref{basic}, \eqref{m1} and \eqref{m2}, we have
\begin{equation}\label{m}
	\|\nabla\mathcal{m}\|_1^2
	-\int_{\mathbb{R}^3}\nabla{\rm div}\bm{q}\cdot\nabla\mathcal{m}\ {\rm d}\bm{x}
	-\int_{\mathbb{R}^3}\nabla^2{\rm div}\bm{q}\cdot\nabla^2\mathcal{m}\ {\rm d}\bm{x}
	\lesssim\epsilon_1\|\nabla\mathcal{n}\|.
\end{equation} 
By adding \eqref{m} to \eqref{zong0}, we obtain that
\begin{align}
	\frac{\rm d}{{\rm d}t}\|\nabla(\mathcal{n},\bm{u},\mathcal{m})\|_1^2
	+\|\nabla^2\bm{u}\|_1^2
	+\|\nabla\mathcal{m}\|^2_1
	\lesssim\epsilon_1\left(\|\nabla\bm{u}\|^2
	+\|\nabla\mathcal{n}\|^2_1\right).
\end{align} 
Finally, we carry out the same method that we use in Lemma \ref{kno0} to add the dissipation of $\mathcal{n}$, which completes the proof.
	\end{proof}
\end{lemma}
The estimates for $\bm{q}$ and its higher order derivatives can be obtained in the same way as that when $\kappa\neq0$. Now the proof of Proposition \ref{a priori} has been accomplished.

\section{Spectral analysis of the linearized system }
By applying ${\rm div}$ to $\eqref{cns2}_4$, we have that
\begin{equation}\label{divq}
	{\rm div}\bm{q}=-4\left(1-\Delta\right)^{-1}\Delta\mathcal{m}-\left(1-\Delta\right)^{-1}\Delta\left(\mathcal{m}^4+4\mathcal{m}^3+6\mathcal{m}^2\right).
\end{equation}
Substitute \eqref{divq} into $\eqref{cns2}_3$ to get
\begin{equation}
\left\{
\begin{aligned}
&\mathcal{n}_t+{\rm div}\bm{u}={\rm R_1},\\[4mm]
&\bm{u}_t+R
\nabla\mathcal{n}+R\nabla\mathcal{m}-\mu\Delta\bm{u}-(\mu+\mu')\nabla{\rm div}\bm{u}={\rm R_2},\\[4mm]
&C_v\mathcal{m}_t+R{\rm div}\bm{u}-\kappa\Delta\mathcal{m}-4\left(1-\Delta\right)^{-1}\Delta\mathcal{m}={\rm R_3},
\end{aligned}
\right.
\end{equation}
where
\begin{equation}
\left\{
\begin{aligned}
{\rm R_1}=&-\mathcal{n}{\rm div}\bm{u}-\bm{u}\cdot\nabla\mathcal{n},\\[4mm]
{\rm R_2}=&\ -\bm{u}\cdot\nabla\bm{u}-\frac{\mathcal{n}}{\mathcal{n}+1}\left[\mu\Delta\bm{u}+\left(\mu+\mu'\right)\nabla{\rm div}\bm{u}\right]+R\left(1-\frac{\mathcal{m}+1}{\mathcal{n}+1}\right)\nabla\mathcal{n},\\[3mm]
{\rm R_3}=&-C_v\bm{u}\cdot\nabla\mathcal{m}-R\mathcal{m}{\rm div}\bm{u}-\kappa\frac{\mathcal{n}}{\mathcal{n}+1}\Delta\mathcal{m}+\frac{2\mu}{\mathcal{n}+1}\mathcal{D}(\bm{u}):\mathcal{D}(\bm{u})+\\
&\ \frac{\mu'}{\mathcal{n}+1}\left({\rm div}\bm{u}\right)^2+\frac{\mathcal{n}}{\mathcal{n}+1}{{\rm div}\bm{q}}+\left(1-\Delta\right)^{-1}\Delta\left(\mathcal{m}^4+4\mathcal{m}^3+6\mathcal{m}^2\right).
\end{aligned}
\right.
\end{equation}

For brevity, we use $\bm{U}=[\mathcal{n},\bm{u},\mathcal{m}]$ to denote the solution to the linearized homogeneous system that
\begin{equation}\label{cns3}
\left\{
\begin{aligned}
&\mathcal{n}_t+{\rm div}\bm{u}=0,\\[4mm]
&\bm{u}_t+R\nabla\mathcal{n}+R\nabla\mathcal{m}-\mu\Delta\bm{u}-(\mu+\mu')\nabla{\rm div}\bm{u}=0,\\[4mm]
&C_v\mathcal{m}_t+R{\rm div}\bm{u}-\kappa\Delta\mathcal{m}-4\left(1-\Delta\right)^{-1}\Delta\mathcal{m}=0,
\end{aligned}
\right.
\end{equation}
with initial data
\begin{align} \label{initial3}
(\mathcal{n},\bm{u},\mathcal{m})(0,\bm{x})
=(\mathcal{n}_0,\bm{u}_0,\mathcal{m}_0)(x) \quad {\rm for}\  \ \bm{x}\in\mathbb{R}^3.
\end{align}

In terms of the semigroup theory for evolutionary equation, the solution $\bm{U}=\left[\mathcal{n},\bm{u},\mathcal{m}\right]$ of linear problem \eqref{cns3}--\eqref{initial3} can be expressed as
\begin{equation}\label{3.4}
  \bm{U}_{t}=\bm{B}\bm{U},\qquad\bm{U}(0)=\bm{U}_{0},\qquad t\geq0,
\end{equation}
which gives rise to
\begin{equation}
  \bm{U}(t)=\bm{S}(t)\bm{U}_{0}:=e^{t\bm{B}}\bm{U}_{0},\qquad t\geq0.
\end{equation}
What left is to analyze the properties of the semigroup $\bm{S}(t)$ through the Fourier transform $\bm{A}\left(\bm{\xi}\right)$ of the operator $\bm{B}$. By applying the Fourier transform to system \eqref{3.4}, we obtain
\begin{equation}
\left\{
  \begin{aligned}
  &\widehat{\bm{U}}_t=\bm{A}\left(\bm{\xi}\right)\widehat{\bm{U}},\\
  &\widehat{\bm{U}}(0)=\widehat{\bm{U}_{0}},
  \end{aligned}
\right.
\end{equation}
where
\begin{equation}
  \bm{A}(\bm{\xi})=
    \left(
      \begin{matrix}
        0&-iR\bm{\xi}^T&0\\[2.5mm]
        -iR\bm{\xi}\quad&-\mu|\bm{\xi}|^2\bm{I}-(\mu+\mu')\bm{\xi}\otimes\bm{\xi}\quad&-iR\bm{\xi}\\[2.5mm]
        0&-i\frac{R}{C_v}\bm{\xi}^T&-\frac{\kappa}{C_v}|\bm{\xi}|^2-\frac{4|\xi|^2}{C_v\left(|\xi|^2+1\right)}
      \end{matrix}
    \right),
\end{equation}
and here $\bm{I}$ is identity matrix.

We can express the solution \eqref{3.4} in terms of the semigroup $S(t)$ as
\begin{gather}
  \bm{U}(t)=\bm{S}(t)\bm{U}_{0}=e^{t\bm{B}}\bm{U}:=\mathcal{F}^{-1}\left(e^{t\bm{A}(\bm{\xi})}\widehat{\bm{U}_{0}}\right),\\
  \bm{A}(\bm{\xi})=\widehat{\bm{B}}(\bm{\xi}),\qquad\bm{\xi}\in\mathbb{R}^3.
\end{gather}
The eigenvalues of the matrix $\bm{A}(\bm{\xi})$ can be computed from the determinant
\begin{equation}
  \det\left(\bm{A}(\bm{\xi})-\lambda \bm{I}\right)=
  -\left(\lambda+\mu|\bm{\xi}|^2\right)^2\mathcal{g}\left(\lambda,|\bm{\xi}|^2\right)=0,
\end{equation}
where
\begin{align}\notag
  \mathcal{g}\left(t\lambda,|\bm{\xi}|^2\right)=&
  \lambda^3
  +\left[\frac{\kappa}{C_v}|\xi|^2+(2\mu+\mu')|\xi|^2+\frac{4|\xi|^2}{C_v\left(|\xi|^2+1\right)}\right]\lambda^2\\\notag
  &+\left[(2\mu+\mu')\frac{\kappa}{C_v}|\bm{\xi}|^4+\left(2\mu+\mu'\right)\frac{4|\xi|^4}{C_v\left(|\xi|^2+1\right)}+\left(1+\frac{R^2}{C_v}\right)|\bm{\xi}|^2\right]\lambda\\\label{3.11}
  &+\frac{\kappa}{C_v}|\bm{\xi}|^4+\frac{4|\xi|^2}{C_v\left(|\xi|^2+1\right)},
\end{align}
which implies
\begin{equation}
  \lambda_1=\lambda_1\left(|\bm{\xi}|\right),\quad \lambda_2=\lambda_2\left(|\bm{\xi}|\right),\quad \lambda_3=\lambda_3\left(|\bm{\xi}|\right),\quad
  \lambda_4=-\mu|\bm{\xi}|^2\text{(with multiplicity 2)}.
\end{equation}
We assume that $\lambda_{1,2,3}\neq-(2\mu+\mu')|\bm{\xi}|^2$. (Otherwise, the root $-(2\mu+\mu')|\bm{\xi}|^2$ has higher multiplicity, which,  however, leads to no difficulties for our following analysis.) The corresponding right eigenvectors $R_i(i=1,2,\cdots,5)$ are
\begin{equation*}
\left(
\begin{matrix}
\frac{|\bm{\xi}|^2}{\lambda_1}\\[1.5mm]
i\xi_1\\[1.5mm]
i\xi_2\\[1.5mm]
i\xi_3\\[1.5mm]
\frac{\frac{R}{C_v}|\bm{\xi}|^2}{\frac{\kappa}{C_v}|\bm{\xi}|^2+\frac{4|\xi|^2}{C_v\left(|\xi|^2+1\right)}+\lambda_1}
\end{matrix}
\right)\ 
\left(
\begin{matrix}
\frac{|\bm{\xi}|^2}{\lambda_2}\\[1.5mm]
i\xi_1\\[1.5mm]
i\xi_2\\[1.5mm]
i\xi_3\\[1.5mm]
\frac{\frac{R}{C_v}|\bm{\xi}|^2}{\frac{\kappa}{C_v}|\bm{\xi}|^2+\frac{4|\xi|^2}{C_v\left(|\xi|^2+1\right)}+\lambda_2}
\end{matrix}
\right)\ 
\left(
\begin{matrix}
\frac{|\bm{\xi}|^2}{\lambda_3}\\[1mm]
i\xi_1\\[1mm]
i\xi_2\\[1mm]
i\xi_3\\[1mm]
\frac{\frac{R}{C_v}|\bm{\xi}|^2}{\frac{\kappa}{C_v}|\bm{\xi}|^2+\frac{4|\xi|^2}{C_v\left(|\xi|^2+1\right)}+\lambda_3}
\end{matrix}
\right)\ \left(
\begin{matrix}
0\\[1mm]
\xi_3\\[1mm]
0\\[1mm]
-\xi_1\\[1mm]
0
\end{matrix}
\right)\ \left(
\begin{matrix}
0\\[1mm]
-\frac{\xi_1\xi_2\xi_3}{\xi_1^2+\xi_3^2}\\[1mm]
\xi_3\\[1mm]
-\frac{\xi_2\xi_3^2}{\xi_1^2+\xi_3^2}\\[1mm]
0
\end{matrix}
\right)
\end{equation*}
and the corresponding left eigenvectors $L_i^T(i=1,2,\cdots,5)$ are
\begin{equation*}
\left(
\begin{matrix}
\frac{R|\bm{\xi}|^2}{\lambda_1}\\[1mm]
i\xi_1\\[1mm]
i\xi_2\\[1mm]
i\xi_3\\[1mm]
\frac{R|\bm{\xi}|^2}{\frac{\kappa}{C_v}|\bm{\xi}|^2+\frac{4|\xi|^2}{C_v\left(|\xi|^2+1\right)}+\lambda_1}
\end{matrix}
\right)\ 
\left(
\begin{matrix}
\frac{R|\bm{\xi}|^2}{\lambda_2}\\[1mm]
i\xi_1\\[1mm]
i\xi_2\\[1mm]
i\xi_3\\[1mm]
\frac{R|\bm{\xi}|^2}{\frac{\kappa}{C_v}|\bm{\xi}|^2+\frac{4|\xi|^2}{C_v\left(|\xi|^2+1\right)}+\lambda_2}
\end{matrix}
\right)\ 
\left(
\begin{matrix}
\frac{R|\bm{\xi}|^2}{\lambda_3}\\[1mm]
i\xi_1\\[1mm]
i\xi_2\\[1mm]
i\xi_3\\[1mm]
\frac{R|\bm{\xi}|^2}{\frac{\kappa}{C_v}|\bm{\xi}|^2+\frac{4|\xi|^2}{C_v\left(|\xi|^2+1\right)}+\lambda_3}
\end{matrix}
\right)\ 
\left(
\begin{matrix}
0\\[1mm]
\xi_3\\[1mm]
0\\[1mm]
-\xi_1\\[1mm]
0
\end{matrix}
\right)\ 
\left(
\begin{matrix}
0\\[1mm]
-\frac{\xi_1\xi_2\xi_3}{\xi_1^2+\xi_3^2}\\[1mm]
\xi_3\\[1mm]
-\frac{\xi_2\xi_3^2}{\xi_1^2+\xi_3^2}\\[1mm]
0
\end{matrix}
\right)
\end{equation*}
If $\lambda_1\neq\lambda_2\neq\lambda_3$, the semigroup $\widehat{\bm{G}}(t,\bm{\xi})=e^{t\bm{A}(\bm{\xi})}$ is expressed as
\begin{align}
  e^{t\bm{A}(\bm{\xi})}
  &=\left(R_1,R_2,R_3,R_4,R_5\right)
\left(  
  \begin{matrix}
  e^{\lambda_1t}&\ &\ &\ &\ \\
  \ &e^{\lambda_2t}&\ &\ &\ \\
  \ &\ &e^{\lambda_3t}&\ &\ \\
  \ &\ &\ &e^{-(2\mu+\mu')|\bm{\xi}|^2}&\ \\
  \ &\ &\ &\ &e^{-(2\mu+\mu')|\bm{\xi}|^2}
  \end{matrix}
  \right) 
  \left(
   \begin{matrix}
     \frac{L_1}{L_1\cdot R_1}\\[1mm] \frac{L_2 }{L_2 \cdot R_2 }\\[1mm]  \frac{L_3 }{L_3 \cdot R_3 }\\[1mm]  \frac{L_4}{L_4\cdot R_4}\\[1mm]  \frac{L_5}{L_5\cdot R_5}
   \end{matrix}
  \right)\\
  &=e^{\lambda_1t}\bm{P}_1(\bm{\xi})
+e^{\lambda_2t}\bm{P}_2(\bm{\xi})
+e^{\lambda_3t}\bm{P}_3(\bm{\xi})
+e^{\lambda_4t}\bm{P}_4(\bm{\xi}),
\end{align}
where the project operators $\bm{P}_i(\bm{\xi})\ (i=1,\ 2,\ 3,\ 4)$ can be computed as
\begin{eqnarray*}
  \bm{P}_1(\bm{\xi})=
  &\left(
    \begin{matrix}
    	\frac{R|\bm{\xi}|^2\left(\frac{\kappa}{C_v}|\bm{\xi}|^2+\frac{4|\xi|^2}{C_v\left(|\xi|^2+1\right)}+\lambda_1\right)}{\lambda_1(\lambda_1-\lambda_2)(\lambda_3-\lambda_1)}
    	&\frac{i\left(\frac{\kappa}{C_v}|\bm{\xi}|^2+\frac{4|\xi|^2}{C_v\left(|\xi|^2+1\right)}+\lambda_1\right)\bm{\xi}^T}{(\lambda_1-\lambda_2)(\lambda_3-\lambda_1)}
    	&\frac{R|\bm{\xi}|^2}{(\lambda_1-\lambda_2)(\lambda_3-\lambda_1)}\\[2mm]
    	
    	\frac{iR\left(\frac{\kappa}{C_v}|\bm{\xi}|^2+\frac{4|\xi|^2}{C_v\left(|\xi|^2+1\right)}+\lambda_1\right)\bm{\xi}}{(\lambda_1-\lambda_2)(\lambda_3-\lambda_1)}\
    	&\frac{\left(\frac{\kappa}{C_v}|\bm{\xi}|^2+\frac{4|\xi|^2}{C_v\left(|\xi|^2+1\right)}+\lambda_1\right)\lambda_1}{(\lambda_1-\lambda_2)(\lambda_1-\lambda_3)}\frac{\bm{\xi}\bm{\xi}^T}{|\bm{\xi}|^2}\
    	&\frac{iR\lambda_1\bm{\xi}}{(\lambda_1-\lambda_2)(\lambda_3-\lambda_1)}\\[2mm]
    	
    	\frac{\frac{R^2}{C_v}|\bm{\xi}|^2}{(\lambda_1-\lambda_2)(\lambda_3-\lambda_1)}
    	&\frac{i\frac{R}{C_v}\lambda_1\bm{\xi}^T}{(\lambda_1-\lambda_2)(\lambda_3-\lambda_1)}
    	&\frac{\lambda_1^2+(2\mu+\mu')|\bm{\xi}|^2\lambda_1+|\bm{\xi}|^2}{(\lambda_1-\lambda_2)(\lambda_1-\lambda_3)}
    	
    \end{matrix}
  \right),\\[4mm]
  \bm{P}_2(\bm{\xi})=
   &\left(
     \begin{matrix}
 	\frac{R|\bm{\xi}|^2\left(\frac{\kappa}{C_v}|\bm{\xi}|^2+\frac{4|\xi|^2}{C_v\left(|\xi|^2+1\right)}+\lambda_2\right)}{\lambda_2(\lambda_1-\lambda_2)(\lambda_2-\lambda_3)}
	&\frac{i\left(\frac{\kappa}{C_v}|\bm{\xi}|^2+\frac{4|\xi|^2}{C_v\left(|\xi|^2+1\right)}+\lambda_2\right)\bm{\xi}^T}{(\lambda_1-\lambda_2)(\lambda_2-\lambda_3)}
	&\frac{R|\bm{\xi}|^2}{(\lambda_1-\lambda_2)(\lambda_2-\lambda_3)}\\[2mm]
	
	\frac{iR\left(\frac{\kappa}{C_v}|\bm{\xi}|^2+\frac{4|\xi|^2}{C_v\left(|\xi|^2+1\right)}+\lambda_2\right)\bm{\xi}}{(\lambda_1-\lambda_2)(\lambda_2-\lambda_3)}\
	&\frac{\left(\frac{\kappa}{C_v}|\bm{\xi}|^2+\frac{4|\xi|^2}{C_v\left(|\xi|^2+1\right)}+\lambda_2\right)\lambda_2}{(\lambda_2-\lambda_1)(\lambda_2-\lambda_3)}\frac{\bm{\xi}\bm{\xi}^T}{|\bm{\xi}|^2}\
	&\frac{iR\lambda_2\bm{\xi}}{(\lambda_1-\lambda_2)(\lambda_2-\lambda_3)}\\[2mm]
	
	\frac{\frac{R^2}{C_v}|\bm{\xi}|^2}{(\lambda_1-\lambda_2)(\lambda_2-\lambda_3)}
	&\frac{i\frac{R}{C_v}\lambda_2\bm{\xi}^T}{(\lambda_1-\lambda_2)(\lambda_2-\lambda_3)}
	&\frac{\lambda_2^2+(2\mu+\mu')|\bm{\xi}|^2\lambda_2+|\bm{\xi}|^2}{(\lambda_2-\lambda_1)(\lambda_2-\lambda_3)}
	
    \end{matrix}
  \right),\\[4mm]
\bm{P}_3(\bm{\xi})=
&\left(
\begin{matrix}
	\frac{R|\bm{\xi}|^2\left(\frac{\kappa}{C_v}|\bm{\xi}|^2+\frac{4|\xi|^2}{C_v\left(|\xi|^2+1\right)}+\lambda_3\right)}{\lambda_3(\lambda_3-\lambda_1)(\lambda_2-\lambda_3)}
	&\frac{i\left(\frac{\kappa}{C_v}|\bm{\xi}|^2+\frac{4|\xi|^2}{C_v\left(|\xi|^2+1\right)}+\lambda_3\right)\bm{\xi}^T}{(\lambda_3-\lambda_1)(\lambda_2-\lambda_3)}
	&\frac{R|\bm{\xi}|^2}{(\lambda_3-\lambda_1)(\lambda_2-\lambda_3)}\\[2mm]
	
\frac{iR\left(\frac{\kappa}{C_v}|\bm{\xi}|^2+\frac{4|\xi|^2}{C_v\left(|\xi|^2+1\right)}+\lambda_3\right)\bm{\xi}}{(\lambda_3-\lambda_1)(\lambda_2-\lambda_3)}\
	&\frac{\left(\frac{\kappa}{C_v}|\bm{\xi}|^2+\frac{4|\xi|^2}{C_v\left(|\xi|^2+1\right)}+\lambda_3\right)\lambda_2}{(\lambda_1-\lambda_3)(\lambda_2-\lambda_3)}\frac{\bm{\xi}\bm{\xi}^T}{|\bm{\xi}|^2}\
	&\frac{iR\lambda_3\bm{\xi}}{(\lambda_3-\lambda_1)(\lambda_2-\lambda_3)}\\[2mm]
	
	\frac{\frac{R^2 }{C_v}|\bm{\xi}|^2}{(\lambda_3-\lambda_1)(\lambda_2-\lambda_3)}
	&\frac{i\frac{R}{C_v}\lambda_3\bm{\xi}^T}{(\lambda_3-\lambda_1)(\lambda_2-\lambda_3)}
	&\frac{\lambda_3^2+(2\mu+\mu')|\bm{\xi}|^2\lambda_3+|\bm{\xi}|^2}{(\lambda_1-\lambda_3)(\lambda_2-\lambda_3)}
	
\end{matrix}
\right),
\end{eqnarray*}

\begin{eqnarray*}
	\bm{P}_4(\bm{\xi})=&
	\left(
	\begin{matrix}
		0&0&0&0&0\\
		0&\frac{\xi_3^2}{\xi_1^2+\xi_3^2}&0&-\frac{\xi_1\xi_3}{\xi_1^2+\xi_3^2}&0\\
		0&0&0&0&0\\
		0&-\frac{\xi_1\xi_3}{\xi_1^2+\xi_3^2}&0&\frac{\xi_1^2}{\xi_1^2+\xi_3^2}&0\\
		0&0&0&0&0
	\end{matrix}
	\right)+
	\left(
	\begin{matrix}
		0&0&0&0&0\\
		0&\frac{\xi_1^2\xi_2^2}{|\bm{\xi}|^2(\xi_1^2+\xi_3^2)}&-\frac{\xi_1\xi_2}{|\bm{\xi}|^2}&\frac{\xi_1\xi_2^2\xi_3}{|\bm{\xi}|^2(\xi_1^2+\xi_3^2)}&0\\
		0&-\frac{\xi_2\xi_1}{|\bm{\xi}|^2}&\frac{\xi_1^1+\xi_3^2}{|\bm{\xi}|^2}&-\frac{\xi_2\xi_3}{|\bm{\xi}|^2}&0\\
		0&\frac{\xi_1\xi_2^2\xi_3}{|\bm{\xi}|^2(\xi_1^2+\xi_3^2)}&-\frac{\xi_2\xi_3}{|\bm{\xi}|^2}&\frac{\xi_2^2\xi_3^2}{|\bm{\xi}|^2(\xi_1^2+\xi_3^2)}&0\\
		0&0&0&0&0
	\end{matrix}
	\right)\\[4mm]
	=&\left(
	\begin{matrix}
		0&\bm{0}_{1\times3}&0\\[2mm]
		
		\bm{0}_{3\times1}\
		&\bm{I}-\frac{\bm{\xi}\bm{\xi}^T}{|\bm{\xi}|^2}\
		&\bm{0}_{3\times1}\\[2mm]
		
		0&\bm{0}_{1\times3}&0
	\end{matrix}
	\right).\qquad\qquad\qquad\qquad\qquad\qquad\qquad\quad\quad
\end{eqnarray*}

Here we have applied the following identities that
\begin{eqnarray*}
	L_1\cdot R_1&&=|\bm{\xi}|^2\left(\frac{|\bm{\xi}|^2}{\lambda_1^2}-1+\frac{\frac{R^2}{C_v}|\bm{\xi}|^2\lambda_1}{\lambda_1\left(\frac{\kappa}{C_v}|\bm{\xi}|^2+\frac{4|\xi|^2}{C_v\left(|\xi|^2+1\right)}+\lambda_1\right)^2}\right)\\
	  &&=|\bm{\xi}|^2\left(\frac{|\bm{\xi}|^2}{\lambda_1^2}-1+\frac{\lambda_1^2+(2\mu+\mu')|\bm{\xi}|^2\lambda_1+|\bm{\xi}|^2}{\lambda_1\left(\frac{\kappa}{C_v}|\bm{\xi}|^2+\frac{4|\xi|^2}{C_v\left(|\xi|^2+1\right)}+\lambda_1\right)^2}\right)\\
	  &&=|\bm{\xi}|^2\frac{\frac{\frac{\kappa}{C_v}|\bm{\xi}|^4}{\lambda_1}-\left(2\mu+\mu'+\frac{\kappa}{C_v}\right)|\bm{\xi}|^2\lambda_1-2\lambda_1^2}{\lambda_1\left(\frac{\kappa}{C_v}|\bm{\xi}|^2+\frac{4|\xi|^2}{C_v\left(|\xi|^2+1\right)}+\lambda_1\right)}
	  =\frac{|\bm{\xi}|^2\left(\lambda_1-\lambda_2\right)\left(\lambda_3-\lambda_1\right)}{\lambda_1\left(\frac{\kappa}{C_v}|\bm{\xi}|^2+\frac{4|\xi|^2}{C_v\left(|\xi|^2+1\right)}+\lambda_1\right)}
\end{eqnarray*}
and $L_i\cdot R_i(i=2,3)$, which can be obtained in the similar way.
From \eqref{3.11}, we can derive the estimates of the eigenvalues as follows by applying the implicit Function Theorem.

\begin{lemma}\label{lem1}
	For a sufficiently small $\epsilon>0$ and a sufficiently large $K>0$, we have
	 \begin{itemize}
	 	\item[$(i)$] when $|\bm{\xi}|<\epsilon$, $\lambda_i\ (i=1,2,3,4)$ has the following expansion that
	 	  \begin{equation*}
            \begin{cases}
              &\lambda_1=-\frac{\kappa+4}{R+C_v}|\bm{\xi}|^2
              +\sum_{j=2}^{\infty}a_j|\bm{\xi}|^{2j},\\[2mm]
              &\lambda_2=-\left(\frac{2\mu+\mu'}{2}+\frac{\kappa+4}{2C_v}\frac{R}{C_v+R}\right)|\bm{\xi}|^2
              +\sum_{j=2}^{\infty}b_{2j}|\bm{\xi}|^{2j}+\\[2mm]
              &\qquad\ \ i\left(\sqrt{R+\frac{R^2}{C_v}}|\bm{\xi}|+\sum_{j=2}^{\infty}b_{2j-1}|\bm{\xi}|^{2j-1}\right),\\[2mm]
              &\lambda_3=\ \lambda_2^*\quad (\text{complex conjugate}\ )\ ,\\[2mm]
              &\lambda_4=-\mu|\bm{\xi}|^2\quad (\text{with multiplicity 2}\ )\ ;
            \end{cases}
	 	  \end{equation*}
	 	  \item[$(ii)$] when $\epsilon\leq|\bm{\xi}|\leq K$, $\lambda_i\ (i=1,2,3,4)$ has the following spectrum gap property that
	 	  \begin{equation*}
	 	  Re(\lambda_j)\leq-C,\qquad\text{for some constant $C>0$}\ ;
	 	  \end{equation*}
	 	\item[$(iii)$] when $|\bm{\xi}|>K$,  $\lambda_i\ (i=1,2,3,4)$ has the following expansion that
	 	  \begin{equation}
	 	    \begin{cases}
	 	      &\lambda_1=-\frac{R}{2\mu+\mu'}+\sum_{j=1}^{\infty}c_j|\bm{\xi}|^{-2j},\\[2mm]
	 	      &\lambda_2=-(2\mu+\mu')|\bm{\xi}|^2+\sum_{j=0}^{\infty}d_j|\bm{\xi}|^{-2j},\\[2mm]
	 	      &\lambda_3=-\frac{\kappa}{C_v}|\bm{\xi}|^2+\sum_{j=0}^{\infty}e_j|\bm{\xi}|^{-2j},\\[2mm]
	 	      &\lambda_4=-\mu|\bm{\xi}|^2\quad(\text{with multiplicity 2}\ )\ .
	 	    \end{cases}
	 	  \end{equation}
	 \end{itemize}
$($Here $a_j,\ b_j,\ c_j,\ d_j,\ e_j$ are real constants.$)$
\end{lemma}
We omit the proof here, which is similar to that of Theorem $2.1-2.3$ in \cite{MR2164944}.

From the estimates above, we have it for a sufficiently small $|\bm{\xi}|$ and some positive constants $d_1$, $d_2$ that
\begin{align}\notag
  &P_2(\bm{\xi})e^{\lambda_2t}+P_3(\bm{\xi})e^{\lambda_3t}\\\notag
  \sim& P_2(\bm{\xi})e^{-d_1|\bm{\xi}|^2t}\left[cos\left(d_2|\bm{\xi}|t\right)+isin\left(d_2|\bm{\xi}|t\right)\right]+P_3(\bm{\xi})e^{-d_1|\bm{\xi}|^2t}\left[cos\left(d_2|\bm{\xi}|t\right)-isin\left(d_2|\bm{\xi}|t\right)\right]\\\label{P23}
  \sim& e^{-d_1|\bm{\xi}|^2t}\left[cos\left(d_2|\bm{\xi}|t\right)\left(P_2+P_3\right)+isin\left(d_2|\bm{\xi}|t\right)(P_2-P_3)\right].
\end{align}
So we need to derive the following lemma by applying the estimates for the eigenvalues above and the proofs are same as \cite{MR2737817,MR3832845}, we omit it here.
\begin{lemma}\label{lem2}
	For a sufficiently small $|\bm{\xi}|$, we have
    the following estimates for $\bm{P}_i(\bm{\xi})\ (i=1,\ 2,\ 3)$ that
	\begin{eqnarray*}
		\bm{P}_1(\bm{\xi})=
		&\left(
		\begin{matrix}
			O(1)&iO(\bm{\xi}^T)&O(1)\\[2mm]
			
			iO(\bm{\xi})
			&O(\bm{\xi}\bm{\xi}^T)
			&iO(\bm{\xi})\\[2mm]
			
			\quad O\left(1\right)\quad
			&\quad iO\left(\bm{\xi}^T\right)\quad
			&\quad O(1)\quad
		\end{matrix}
		\right),\\[3mm]
		\bm{P}_2(\bm{\xi})+\bm{P}_3(\bm{\xi})=
		&\left(
		\begin{matrix}
			O(1)&iO(\bm{\xi}^T)&O(1)\\[2mm]
			
			iO(\bm{\xi})
			&O(\frac{\bm{\xi}\bm{\xi}^T}{|\bm{\xi}|^2})
			&iO(\bm{\xi})\\[2mm]
			
			\quad O\left(1\right)\quad
			&\quad iO\left(\bm{\xi}^T\right)\quad
			&\quad O(1)\quad
		\end{matrix}
		\right),\\[3mm]
		\bm{P}_2(\bm{\xi})-\bm{P}_3(\bm{\xi})=
		&\left(
		\begin{matrix}
			iO(|\bm{\xi}|)&O(\frac{\bm{\xi}^T}{|\bm{\xi}|})&iO(|\bm{\xi}|)\\[2mm]
			
			O(\frac{\bm{\xi}}{|\bm{\xi}|})
			&iO(\frac{\bm{\xi}\bm{\xi}^T}{|\bm{\xi}|})
			&O(\frac{\bm{\xi}}{|\bm{\xi}|})\\[2mm]
			
			\quad iO\left(|\bm{\xi}|\right)\quad
			&\quad O\left(\frac{\bm{\xi}^T}{|\bm{\xi}|}\right)\quad
			&\quad iO(|\bm{\xi}|)\quad
		\end{matrix}
		\right).      
	\end{eqnarray*}

\end{lemma}
Similarly, we have the following two lemmas
\begin{lemma}\label{lem3}
	When $|\bm{\xi}|\geq K$, we have
	$\bm{P}_i(\bm{\xi})\ (i=1,\ 2,\ 3)$ have the following estimates:
	\begin{eqnarray*}
		\bm{P}_1(\bm{\xi})&=
		\left(
			\begin{matrix}
			O(1)&iO(\frac{\bm{\xi}^T}{|\bm{\xi}|^2})&O(\frac{1}{|\bm{\xi}|^2})\\[2mm]
			
			iO(\frac{\bm{\xi}}{|\bm{\xi}|^2})
			&O(\frac{\bm{\xi}\bm{\xi}^T}{|\bm{\xi}|^4})
			&iO(\frac{\bm{\xi}}{|\bm{\xi}|^4})\\[2mm]
			
			\quad O\left(\frac{1}{|\bm{\xi}|^2}\right)\quad
			&\quad iO\left(\frac{\bm{\xi}^T}{|\bm{\xi}|^4}\right)\quad
			&\quad O(\frac{1}{|\bm{\xi}|^2})\quad
		\end{matrix}
		\right),\\[3mm]
		\bm{P}_2(\bm{\xi})&=
		\left(
			\begin{matrix}
			O(\frac{1}{|\bm{\xi}|^4})&iO(\frac{\bm{\xi}^T}{|\bm{\xi}|^4})&O(\frac{1}{|\bm{\xi}|^2})\\[2mm]
			
			iO(\frac{\bm{\xi}}{|\bm{\xi}|^4})
			&O(\frac{\bm{\xi}\bm{\xi}^T}{|\bm{\xi}|^4})
			&iO(\frac{\bm{\xi}}{|\bm{\xi}|^2})\\[2mm]
			
			\quad O\left(\frac{1}{|\bm{\xi}|^2}\right)\quad
			&\quad iO\left(\frac{\bm{\xi}^T}{|\bm{\xi}|^2}\right)\quad
			&\quad O(1)\quad
		\end{matrix}
		\right),\\[3mm]
		\bm{P}_3(\bm{\xi})&=
		\left(
		\begin{matrix}
			O(\frac{1}{|\bm{\xi}|^2})&iO(\frac{\bm{\xi}^T}{|\bm{\xi}|^2})&O(\frac{1}{|\bm{\xi}|^2})\\[2mm]
			
			iO(\frac{\bm{\xi}}{|\bm{\xi}|^2})
			&O(\frac{\bm{\xi}\bm{\xi}^T}{|\bm{\xi}|^2})
			&iO(\frac{\bm{\xi}}{|\bm{\xi}|^2})\\[2mm]
			
			\quad O\left(\frac{1}{|\bm{\xi}|^2}\right)\quad
			&\quad iO\left(\frac{\bm{\xi}^T}{|\bm{\xi}|^2}\right)\quad
			&\quad O(1)\quad
		\end{matrix}
		\right),\\[2mm]	
	\end{eqnarray*}
\end{lemma}

\begin{lemma}\label{lem5}
	When $\epsilon \leq|\bm{\xi}|\leq K$, we have it for some positive constant $c$ that
\begin{equation}
	\lrn e^{t\bm{A}(\bm{\xi})}\rrn\leq ce^{-ct}
\end{equation}
\begin{proof}
	When $\epsilon \leq|\bm{\xi}|\leq K$, there are some complicated cases, like we cannot find five linearly independent eigenvectors or there are changes in the multiplicity of eigenvalues. However, we can see that we already have two linearly independent eigenvectors for $\lambda=-(2\mu+\mu')|\bm{\xi}|^2$. Even in the most extreme case that $\lambda=-(2\mu+\mu')|\bm{\xi}|^2$ with multiplicity $5$ has no other linearly independent eigenvectors, the Jordan normal form of $\bm{A}(\bm{\xi})$ has a $4\times4$ Jordan block and we can still have it for some positive constant $c$ that
	\begin{equation*}
		\lrn e^{t\bm{A}(\bm{\xi})}\rrn\lesssim(1+t)^3e^{-Ct}\leq ce^{-ct},
	\end{equation*}
	where $C$ is the same constant in Lemma \ref{lem1} $(ii)$.
\end{proof}
\end{lemma}
Next, we will obtain some refined $L^p-L^q$ time-decay properties for $\bm{U}=(\mathcal{n}, \bm{u}, \mathcal{m})$. To achieve that, we first make the time-frequency pointwise estimates for $\widehat{\mathcal{n}},\ \widehat{\bm{u}},\ \widehat{\mathcal{m}}$ as follows.

\begin{lemma}\label{lem6}
	Let $\bm{U}=(\mathcal{n}, \bm{u}, \mathcal{m})$ be the solution to the linearized homogeneous system \eqref{cns3}--\eqref{initial3}. Then there exist constants $\epsilon>0,\ c>0,\ C>0$ such that for all $t>0$, $|\bm{\xi}|\leq\epsilon$,
	  \begin{equation}\label{3.14}
	    \left|\left(\widehat{\mathcal{n}}, \widehat{\bm{u}}, \widehat{\mathcal{m}}\right)(t,\bm{\xi})\right|
	    \leq c\exp\left(-c|\bm{\xi}|^2t\right)\left|\left(\widehat{\mathcal{n}_0}, \widehat{\bm{u}_{0}}, \widehat{\mathcal{m}_0}\right)(\bm{\xi})\right|
	  \end{equation}
	and for all $t>0$, $|\bm{\xi}|\geq\epsilon$,
    \begin{equation}\label{3.15}
	\left|\left(\widehat{\mathcal{n}}, \widehat{\bm{u}}, \widehat{\mathcal{m}}\right)(t,\bm{\xi})\right|
	\leq c\exp\left(-ct\right)\left|\left(\widehat{\mathcal{n}_0}, \widehat{\bm{u}_{ 0}}, \widehat{\mathcal{m}_0}\right)(\bm{\xi})\right|
	\end{equation}
\end{lemma}
\begin{proof}
	In order to get the upper bound of $\left(\widehat{\mathcal{n}}, \widehat{\bm{u}}, \widehat{\mathcal{m}}\right)(t,\bm{\xi})$, we need to obtain the estimate for $\widehat{\bm{G}}(t,\bm{\xi})$, which is the Fourier transform of Green function $\bm{G}^{t\bm{B}}(t,\bm{x})$. By direct computation, we can verify the exact expression of $\widehat{\bm{G}}(t,\bm{\xi})$, which is that
	  \begin{equation*}
	    \widehat{\bm{G}}(t,\bm{\xi})=e^{t\bm{A}(\bm{\xi})}=e^{\lambda_1t}\bm{P}_1+e^{\lambda_2t}\bm{P}_2+e^{\lambda_3t}\bm{P}_3+e^{\lambda_4t}\bm{P}_4.
	  \end{equation*}
	For each element of $\widehat{\bm{G}}(t,\bm{\xi})$, by applying Lemma \ref{lem1}--\ref{lem5}, we have
\begin{align}
	&\left|\widehat{G_{ij}}\right|\leq ce^{-c|\bm{\xi}|^2t},\quad 1\leq i,j\leq5,\ \text{as}\ |\bm{\xi}|\leq\epsilon,\label{3.16}\\[2mm]
	&\left|\widehat{G_{ij}}\right|\leq ce^{-ct},\quad 1\leq i,j\leq5,\ \text{as}\ \epsilon\leq|\bm{\xi}|\leq K,\label{3.17}\\[2mm]
	&\left|\widehat{G_{ij}}\right|\leq ce^{-ct},\quad 1\leq i,j\leq5,\ \text{as}\ |\bm{\xi}|\geq K.\label{3.18}
\end{align}
Thus, from \eqref{3.16}, we have it for $|\bm{\xi}|\leq\epsilon$ that
\begin{align*}
\left|\widehat{\mathcal{n}}(t,\bm{\xi})\right|
	&=\left|\widehat{G_{11}}\widehat{\mathcal{n}_0}(\bm{\xi})+\left(\widehat{G_{12}},\widehat{G_{13}},\widehat{G_{14}}\right)\cdot\widehat{\bm{u}_{ 0}}(\bm{\xi})+\widehat{G_{15}}\widehat{\mathcal{m}_0}(\bm{\xi})\right|\\
	&\leq\left|\widehat{G_{11}}\right|\left|\widehat{\mathcal{n}_0}(\bm{\xi})\right|+\left|\left(\widehat{G_{12}},\widehat{G_{13}},\widehat{G_{14}}\right)\right|\left|\widehat{\bm{u}_{ 0}}(\bm{\xi})\right|+\left|\widehat{G_{15}}\right|\left|\widehat{\mathcal{m}_0}(\bm{\xi})\right|\\
	&\leq ce^{-c|\bm{\xi}|^2t}\left|\left[\widehat{\mathcal{n}_0}(\bm{\xi}),\widehat{\bm{u}_{ 0}}(\bm{\xi}),\widehat{\mathcal{m}_0}(\bm{\xi})\right]\right|,\\[3mm]
\left|\widehat{\bm{u}}(t,\bm{\xi})\right|
	&=\left|\left(\widehat{G_{i1}}\right)\widehat{\mathcal{n}_0}(\bm{\xi})+\left(\widehat{G_{ij}}\right)\cdot\widehat{\bm{u}_{ 0}}(\bm{\xi})+\left(\widehat{G_{i5}}\right)\widehat{\mathcal{m}_0}(\bm{\xi})\right|\\
	&\leq\left|\left(\widehat{G_{i1}}\right)\right|\left|\widehat{\mathcal{n}_0}(\bm{\xi})\right|+\left|\left(\widehat{G_{ij}}\right)\right|\left|\widehat{\bm{u}_{ 0}}(\bm{\xi})\right|+\left|\left(\widehat{G_{i5}}\right)\right|\left|\widehat{\mathcal{m}_0}(\bm{\xi})\right|\\
	&\leq ce^{-c|\bm{\xi}|^2t}\left|\left[\widehat{\mathcal{n}_0}(\bm{\xi}),\widehat{\bm{u}_{ 0}}(\bm{\xi}),\widehat{\mathcal{m}_0}(\bm{\xi})\right]\right|\qquad(i,j=2,3,4),\\[3mm]
\left|\widehat{\mathcal{m}}(t,\bm{\xi})\right|
	&=\left|\widehat{G_{51}}\widehat{\mathcal{n}_0}(\bm{\xi})+\left(\widehat{G_{52}},\widehat{G_{53}},\widehat{G_{54}}\right)\cdot\widehat{\bm{u}_{ 0}}(\bm{\xi})+\widehat{G_{55}}\widehat{\mathcal{m}_0}(\bm{\xi})\right|\\
	&\leq\left|\widehat{G_{51}}\right|\left|\widehat{\mathcal{n}_0}(\bm{\xi})\right|+\left|\left(\widehat{G_{52}},\widehat{G_{53}},\widehat{G_{54}}\right)\right|\left|\widehat{\bm{u}_{ 0}}(\bm{\xi})\right|+\left|\widehat{G_{55}}\right|\left|\widehat{\mathcal{m}_0}(\bm{\xi})\right|\\
	&\leq ce^{-c|\bm{\xi}|^2t}\left|\left[\widehat{\mathcal{n}_0}(\bm{\xi}),\widehat{\bm{u}_{ 0}}(\bm{\xi}),\widehat{\mathcal{m}_0}(\bm{\xi})\right]\right|,
\end{align*}
which prove \eqref{3.14}. By applying the same method, \eqref{3.15} follows from \eqref{3.17} and \eqref{3.18} directly. This completes the proof of Lemma \ref{lem6}.
\end{proof}

\begin{proposition}\label{th3.6}
Suppose that $\bm{U}=[\mathcal{n},\bm{u},\mathcal{m}]$ is the solution to the Cauchy problem \eqref{cns3}--\eqref{initial3}. Then, for any integer $k\geq 0$, $\bm{U}$ satisfies the following time-decay property that
	\begin{eqnarray*}
\left\|\nabla^k(\mathcal{n},\bm{u},\mathcal{m})(t)\right\|_{L^2}&&
      \leq c_k(1+t)^{-\frac{3}{4}-\frac{k}{2}}\left(\left\|(\mathcal{n}_0,\bm{u}_{ 0},\mathcal{m}_0)\right\|_{L^1}
	  +\left\|\nabla^{k}(\mathcal{n}_0,\bm{u}_{ 0},\mathcal{m}_0)\right\|\right)\ (k\geq0),\\
	  \left\|\nabla^k(\mathcal{n},\bm{u},\mathcal{m})(t)\right\|_{L^2}&&
	  \leq c'_k(1+t)^{-\frac{3}{2}}\left(\left\|\nabla^{k-3}(\mathcal{n}_0,\bm{u}_{ 0},\mathcal{m}_0)\right\|
	  +\left\|\nabla^{k}(\mathcal{n}_0,\bm{u}_{ 0},\mathcal{m}_0)\right\|\right)\ (k\geq3).
  	\end{eqnarray*}
\end{proposition}
\begin{proof}
\begin{equation}\label{1}
\begin{aligned}
 \left\|\nabla^k \bm{u}(t)\right\|_{L^2}&\lesssim \||\bm{\xi}|^k\widehat{\bm{u}}(t)\|_{L^{2}}\\[2mm]
    &\lesssim \||\bm{\xi}|^k\widehat{\bm{u}}(t)\|_{L^{2}(|\bm{\xi}|\leq \epsilon)}+ \||\bm{\xi}|^k\widehat{\bm{u}}(t)\|_{L^{2}(|\bm{\xi}|\geq \epsilon)},
\end{aligned}
\end{equation}

From Lemma \ref{lem6}, on the one hand, it holds for any $t>1$ that
\begin{eqnarray*}\label{2}
         \||\bm{\xi}|^k\widehat{\bm{u}}(t)\|_{L^{2}(|\bm{\xi}|\leq\epsilon)}
 &\lesssim&\left(\int_{\{|\bm{\xi}|\leq\epsilon\}}|\bm{\xi}|^{2k}e^{-2c|\bm{\xi}|^2t}\left|(\widehat{\mathcal{n}_0},\widehat{\bm{u}_{ 0}},\widehat{\mathcal{m}_0})\right|^2d\bm{\xi}\right)^{\frac 12}\\
 &\lesssim&\lrn(\widehat{\mathcal{n}_0},\widehat{\bm{u}_{ 0}},\widehat{\mathcal{m}_0})\rrn\left(\int_{\{|\bm{\xi}|\leq\epsilon\}}|\bm{\xi}|^{2k}e^{-2c|\bm{\xi}|^2t}d\bm{\xi}\right)^{\frac 12}\\
  &\xlongequal{\zeta=t^{\frac{1}{2}}\xi}&\lrn(\widehat{\mathcal{n}_0},\widehat{\bm{u}_{ 0}},\widehat{\mathcal{m}_0})\rrn\left(\int_{\{|\bm{\zeta}|\leq\epsilon t^{\frac 1 2}\}}t^{-k-\frac 32}|\bm{\zeta}|^{2k}e^{-2c|\bm{\zeta}|^2}d\bm{\zeta}\right)^{\frac 12}\\[3mm]
 &\lesssim&(1+t)^{-\frac34-\frac k2}\left\|(\mathcal{n}_0,\bm{u}_{ 0},\mathcal{m}_0)\right\|_{L^1}.
\end{eqnarray*}
On the other hand, it is also true for any $t>1$ when $k\geq3$ that
\begin{eqnarray*}\label{4}
	\||\bm{\xi}|^k\widehat{\bm{u}}(t)\|_{L^{2}(|\bm{\xi}|\leq\epsilon)}
	&\lesssim&\left(\int_{\{|\bm{\xi}|\leq\epsilon\}}|\bm{\xi}|^{2k}e^{-2c|\bm{\xi}|^2t}\left|(\widehat{\mathcal{n}_0},\widehat{\bm{u}_{ 0}},\widehat{\mathcal{m}_0})\right|^2d\bm{\xi}\right)^{\frac 12}\\
	&\lesssim&(1+t)^{-\frac 32}\left|\left(|\bm{\xi}|^2t\right)^{\frac 32}e^{-c|\bm{\xi}|^2t}\right|\left(\int_{\{|\bm{\xi}|\leq\epsilon\}}|\bm{\xi}|^{2k-6}\left|(\widehat{\mathcal{n}_0},\widehat{\bm{u}_{ 0}},\widehat{\mathcal{m}_0})\right|^2d\bm{\xi}\right)^{\frac 12}\\[3mm]
	&\lesssim&(1+t)^{-\frac32}\left\|\nabla^{k-3}(\mathcal{n}_0,\bm{u}_{ 0},\mathcal{m}_0)\right\|.
\end{eqnarray*}
As for high frequency, we have that
\begin{eqnarray*}\label{3}
\||\bm{\xi}|^k\widehat{\bm{u}}(t)\|_{L^{2}(|\bm{\xi}|\geq \epsilon)}&\lesssim& e^{-ct}\||\bm{\xi}|^k(\widehat{\mathcal{n}_0},\widehat{\bm{u}_{ 0}},\widehat{\mathcal{m}_0})\|_{L^{2}(|\bm{\xi}|\geq \epsilon)}\\[2mm]
&\lesssim&e^{-ct}\|\nabla^k(\mathcal{n}_0,\bm{u}_{ 0},\mathcal{m}_0)\|_{L^{2}(|\bm{\xi}|\geq \epsilon)}.
\end{eqnarray*}

From the conclusion above, we get the estimate for $\|\nabla^k\bm{u}\|_{L^2}$. Similarly, the estimates of $\|\nabla^k \mathcal{n}\|_{L^2}$ and $\|\nabla^k \mathcal{m}\|_{L^2}$ can be obtained. This completes the proof of Proposition \ref{th3.6}.
\end{proof}

\begin{proposition}\label{Pro9}
Suppose that $\bm{U}=[\mathcal{n},\bm{u},\mathcal{m}]$ is the solution to the Cauchy problem \eqref{cns3}--\eqref{initial3}.Then there exist two constants $a_1,a_2>0$ such
	that
	\begin{equation}
	a_1 (1+t)^{-\frac34}\leq\|\left(\mathcal{n},\bm{u},\mathcal{m} \right)(t)\|\leq a_2 (1+t)^{-\frac34}.
	\end{equation}	
\end{proposition}

\begin{proof}
Dute to Proposition \ref{th3.6}, we only to show the lower bound of time-decay rates for the solution under the assumptions of \ref{th1.1}. We prove $\|\mathcal{n}\|\geq a_1 (1+t)^{-\frac34}$ first, and the other two estimates are similar.  Firstly, by the Plancherel's theorem, we know
\begin{eqnarray*}
\|\mathcal{n}\|^2=\|\widehat{\mathcal{n}}\|^2&=&\left\|\widehat{G_{11}}\widehat{\mathcal{n}_0}+\left(\widehat{G_{12}},\widehat{G_{13}},\widehat{G_{14}}\right)\cdot\widehat{\bm{u}_{ 0}}+\widehat{G_{15}}\widehat{\mathcal{m}_0}\right\|^2\\[3mm]
&\geq&\int_{|\bm{\xi}|<\epsilon}\left|\widehat{G_{11}'}\widehat{\mathcal{n}_0}+\left(\widehat{G_{12}'},\widehat{G_{13}'},\widehat{G_{14}'}\right)\cdot\widehat{\bm{u}_{ 0}}+\widehat{G_{15}'}\widehat{\mathcal{m}_0}+\widehat{\bm{R}}\cdot\widehat{\bm{U}_0}\right|^2d\bm{\xi}\\[3mm]
&\geq &\frac12\int_{|\bm{\xi}|<\epsilon}\left|\widehat{G_{11}'}\widehat{\mathcal{n}_0}+\left(\widehat{G_{12}'},\widehat{G_{13}'},\widehat{G_{14}'}\right)\cdot\widehat{\bm{u}_{ 0}}+\widehat{G_{15}'}\widehat{\mathcal{m}_0}\right|^2d\bm{\xi}
-\int_{|\bm{\xi}|<\epsilon}\left|\widehat{\bm{R}}\cdot\widehat{\bm{U}_0}\right|^2d\bm{\xi}\\[3mm]
&:=&\frac{1}{2}T_1-T_2.
\end{eqnarray*}
Here,
\begin{equation*}
	\left(\widehat{G_{11}},\widehat{G_{12}},\widehat{G_{13}},\widehat{G_{14}},\widehat{G_{15}}\right)=\left(\widehat{G_{11}'},\widehat{G_{12}'},\widehat{G_{13}'},\widehat{G_{14}'},\widehat{G_{15}'}\right)+\widehat{\bm{R}},
\end{equation*}
and $\widehat{\bm{R}}$ is the remainder terms which are in the same order as $|\bm{\xi}|$ or higher.
From Proposition \ref{th3.6}, it holds that
\begin{eqnarray*}
T_2\leq C(1+t)^{-\frac52} \|\bm{U}_0\|_{L^1}^2.
\end{eqnarray*}
As for $T_1$, we have that
\begin{eqnarray*}
T_1&\geq&\frac12\left\|\widehat{G_{11}'}\widehat{\mathcal{n}_0}\right\|^2-2\left[\left\|\left(\widehat{G_{12}'},\widehat{G_{13}'},\widehat{G_{14}'}\right)\cdot\widehat{\bm{u}_{ 0}}\right\|^2+\left\|\widehat{G_{15}'}\widehat{\mathcal{m}_0}\right\|^2\right]\\[3mm]
&:=&T_3-T_4.
\end{eqnarray*}
It follows from \eqref{P23}, Lemma \ref{lem2} and Proposition \ref{th3.6} for some positive constants $C$ and $c_0'$ that,
\begin{eqnarray*}
T_4\leq C\int_{|\bm{\xi}|<\epsilon}|e^{-d_1|\bm{\xi}|^2t}|^2|\widehat{\bm{u}}_0|^2+|e^{-d_1|\bm{\xi}|^2t}|^2|\widehat{\mathcal{m}}_0|^2d\bm{\xi}\leq c_0'(1+t)^{-\frac32}\|(\bm{u}_0,\mathcal{m}_0)\|_{L^1}^2.	
\end{eqnarray*}

And the term $T_3$ satisfies for $t>1$ that
\begin{eqnarray*}
T_3&=&\frac{1}{2}\int_{|\bm{\xi}|<\epsilon}\left|e^{-\frac{\kappa+4}{R+C_v}|\bm{\xi}|^2t}+e^{-\left(\frac{2\mu+\mu'}{2}+\frac{(\kappa+4)R}{2C_v(C_v+R)}\right)|\bm{\xi}|^2t}cos\left(\sqrt{R+\frac{R^2}{c_v}}|\bm{\xi}|t\right)\right|^2|\widehat{\mathcal{n}}_0|^2d\bm{\xi}\\[3mm]	
&\gtrsim&\|\mathcal{n}_0\|_{L^1}^2\int_{|\bm{\zeta}|<\epsilon \sqrt{t}}\left|e^{-\frac{\kappa+4}{R+C_v}|\bm{\zeta}|^2}+e^{-\left(\frac{2\mu+\mu'}{2}+\frac{(\kappa+4)R}{2C_v(C_v+R)}\right)|\bm{\zeta}|^2}cos\left(\sqrt{R+\frac{R^2}{c_v}}|\bm{\zeta}|\sqrt{t}\right)\right|^2t^{-\frac{3}{2}}d\bm{\zeta}\\[3mm]
&\gtrsim&\|\mathcal{n}_0\|_{L^1}^2(1+t)^{-\frac32}\int_{|\bm{\zeta}|<\epsilon \sqrt{t} } \left|e^{-\frac{\kappa+4}{R+C_v}|\bm{\zeta}|^2}+e^{-\left(\frac{2\mu+\mu'}{2}+\frac{(\kappa+4)R}{2C_v(C_v+R)}\right)|\bm{\zeta}|^2}cos\left(\sqrt{R+\frac{R^2}{c_v}}|\bm{\zeta}|\sqrt{t}\right)\right|^2d\bm{\zeta}\\[3mm]
&:=&C\|\mathcal{n}_0\|_{L^1}^2(1+t)^{-\frac32}\int_{|\bm{\zeta}|<\epsilon \sqrt{t}}|f(|\bm{\zeta}|^2,t)|^2  d\bm{\zeta}	\\[3mm]
&:=&C\|\mathcal{n}_0\|_{L^1}^2(1+t)^{-\frac32}F(t),
\end{eqnarray*}
where one can verify that $f(y,t)$ is continuous with respect to $(y,t)$ and $F(t)$ is a non-negative, continuous  function.For simplicity, let$\frac{\kappa+4}{R+C_v}=\sigma_1$, $\frac{2\mu+\mu'}{2}+\frac{(\kappa+4)R}{2C_v(C_v+R)}=\sigma_2$ and $\sqrt{R+\frac{R^2}{C_v}}=\sigma_3$, then 
\begin{eqnarray*}	
	f({|\bm{\zeta}|^2},t)=e^{-\sigma_1|\bm{\zeta}|^2}+e^{-\sigma_2|\bm{\zeta}|^2}cos\left(\sigma_3|\bm{\zeta}|\sqrt{t}\right).	
\end{eqnarray*}

We claim that there exist $C(\epsilon,K)>0$ such that $F(t)\geq C(\epsilon,K)$ uniformly for $t>0$. Clearly, this is true, when $1<t\leq K$ ( $K$ is a sufficiently large positive constant such that $\epsilon\sqrt{K}-\frac{1}{2\sigma_1}>\frac{2\pi}{\sigma_3\sqrt{K}}$ and we define $M=\epsilon\sqrt{K}$ ). when $t>K$, we just need to show that there exists a nonzero measure set on which the function value of $f(|\bm{\zeta}|^2,t)| $ is not zero. In fact, we have that
\begin{eqnarray}\notag
F(t)&=&\int_{|\bm{\zeta}|\sqrt{t}<\epsilon t } \left|e^{-\sigma_1|\bm{\zeta}|^2}+e^{-\sigma_2|\bm{\zeta}|^2}cos\left(\sigma_3|\bm{\zeta}|\sqrt{t}\right)\right|^2d\bm{\zeta}\\[3mm]\notag
&\geq&\int_{\left\{\zeta\Big|\frac{\sigma_3\sqrt{t}}{2\sigma_1}<\sigma_3|\bm{\zeta}|\sqrt{t}<M\sigma_3\sqrt{t},\  cos\left(\sigma_3|\bm{\zeta}|\sqrt{t}\right)\geq0\right\}} e^{-2\sigma_1|\bm{\zeta}|^2}d\bm{\zeta}\\[3mm]\notag
&\geq&c\int_{\left\{\zeta\Big|\frac{\sigma_3\sqrt{t}}{2\sigma_1}<\sigma_3r\sqrt{t}<M\sigma_3\sqrt{t},\  cos\left(\sigma_3r\sqrt{t}\right)\geq0\right\}} r^2e^{-2\sigma_1r^2}dr\\[3mm]\label{Ft}
&\geq&\frac{M-\frac{1}{2\sigma_1}}{4}M^2e^{-2\sigma_1M^2}=:C(\epsilon,K).
\end{eqnarray}


Thus, we derive that
\begin{eqnarray*}
T_3\geq C'\|\mathcal{n}_0\|_{L^1}^2(1+t)^{-\frac32},
\end{eqnarray*}
where $C'$ depends on $C(\epsilon, K)$ and is independent of $t$.
Then we have that
\begin{align*}
   T_1\geq (1+t)^{-\frac32}\left(C'\|\mathcal{n}_0\|_{L^1}^2-c_0'\|(\bm{u}_0,\mathcal{m}_0)\|_{L^1}^2\right),
\end{align*}
and here we let $c_0=\frac{c_0'}{C'}$. 
From the estimates above, we conclude it for $t> 1$ that
\begin{eqnarray*}
	\|\mathcal{n}\|=\|\widehat{\mathcal{n}}\|\geq a_1(1+t)^{-\frac34}.
\end{eqnarray*}
Similar to $\mathcal{n}$, we can conclude 
\begin{eqnarray*}
	\|\mathcal{m}(t)\|=\|\widehat{\mathcal{m}}(t)\|\geq a_1(1+t)^{-\frac34}.
\end{eqnarray*}	
As for $\bm{u}$, similar to \eqref{Ft}, we need to prove there exists a positive constant $C$ such that 
\begin{equation*}
	\int_{|\bm{\zeta}|\sqrt{t}<\epsilon t } \left|e^{-\sigma_2|\bm{\zeta}|^2}sin\left(\sigma_3|\bm{\zeta}|\sqrt{t}\right)\right|^2d\bm{\zeta}>C.
\end{equation*}
we omit the details here. Since the upper bound is given by Proposition \ref{th3.6}, the proof is completed.
\end{proof}

\section{Decay rate of the nonlinear system}
\begin{lemma}\label{nablak}
If there exists a small enough constant $\epsilon_2>0$ such that the initial data satisfies  
\begin{equation*}
\|(\rho_0-1,\bm{u}_0,\theta_0-1)\|_{L^1\cap H^{k+2}}\leq\epsilon_2,\qquad(\ \text{for any integer}\ k\geq0) 
\end{equation*}
 then there exists a positive constant $C_k$ such that
	\begin{align}\label{k}
	&\|\nabla^i(\mathcal{n},\bm{u},\mathcal{m})(t)\|_2\leq C(1+t)^{-\frac{3}{4}-\frac i2}\|(\mathcal{n}_0,\bm{u}_0,\mathcal{m}_0)\|_{L^1\cap H^{i+2}},\quad(0\leq i\leq k)\\\label{q}
	&\|\nabla^j\bm{q}\|\leq C\|\nabla^{j+1}\mathcal{m}\|\quad(0\leq j\leq k+1),\qquad\text{for any $t\geq0$}.
	\end{align}
\begin{proof}
We define the energy functional
\begin{equation*}
F_k(t)=\|\nabla^k(\mathcal{n},\bm{u},\mathcal{m})(t)\|^2_2
+\delta\sum_{k\leq|\alpha|\leq k+1}\left\langle\partial_{\bm{x}}^{\alpha}\bm{u},\nabla\partial_{\bm{x}}^{\alpha}\mathcal{n}\right\rangle(t).
\end{equation*}
Here $\delta>0$ is a fix small constant which can be decided in the same way as that in Lemma \ref{basiclemma}, \ref{kno0} and \ref{kis0}. It can be noticed that $F_k(t)$ is equivalent to $\|\nabla^k(\mathcal{n},\bm{u},\mathcal{m})(t)\|^2_2$, that is there exists a constant $C>1$, such that 
\begin{equation*}
\frac{1}{C}\|\nabla^k(\mathcal{n},\bm{u},\mathcal{m})(t)\|^2_2\leq F_k(t)\leq C\|\nabla^k(\mathcal{n},\bm{u},\mathcal{m})(t)\|^2_2.
\end{equation*}
We define
\begin{equation*}
G_k(t)=\sup_{0\leq\tau\leq t}\left\{\sum_{i=0}^{k}(1+\tau)^{\frac{3}{2}+i}F_i(\tau)\right\}.
\end{equation*}
Now we use the inductive method to prove that
\begin{align}\label{Gk}
	G_k(t)\lesssim \|(\mathcal{n}_0,\bm{u}_0,\mathcal{m}_0)\|^2_{L^1\cap H^{k+2}}+G_k^2(t).
\end{align}
First, we show that \eqref{Gk} holds for $k=0$. From Duhamel's principle, it holds that
\begin{equation}\label{Duhamel}
(\mathcal{n},\bm{u},\mathcal{m})(t,\bm{x})=e^{t\bm{B}}(\mathcal{n}_0,\bm{u}_0,\mathcal{m}_0)(\bm{x})+\int_{0}^{t}e^{(t-\tau)\bm{B}}\left({\rm R_1},{\rm R_2},{\rm R_3}\right)(\tau){\rm d}\tau,
\end{equation}
where $e^{t\bm{B}}$ is the linearized solution operator. By applying the linear estimate on $(\mathcal{n},\bm{u},\mathcal{m})(t,\bm{x})$ in Proposition \ref{th3.6}, we have
\begin{align}{\label{0jie}}
\|(\mathcal{n},\bm{u},\mathcal{m})(t)\|\leq C\|\bm{U}_0\|_{L^1\cap L^2}(1+t)^{-\frac 34}
+C  \int_{0}^{t}(1+t-\tau)^{-\frac 34}\|({\rm R_1},{\rm R_2},{\rm R_3})(\tau)\|_{L^1\cap L^2} d\tau
\end{align}
The nonlinear source terms can be estimated as follows.
\begin{align}\label{11}
&\|({\rm R_1},{\rm R_2},{\rm R_3})(t)\|_{L^1}
\lesssim\|(\mathcal{n},\bm{u},\mathcal{m})(t)\|_1\|(\mathcal{n},\bm{u},\mathcal{m})(t)\|_2
\lesssim(1+t)^{-\frac 32}G_0(t),\\\label{12}
&\|({\rm R_1},{\rm R_2},{\rm R_3})(t)\|\lesssim\|(\mathcal{n},\bm{u},\mathcal{m})(t)\|^2_2\lesssim(1+t)^{-\frac 32}G_0(t).
\end{align}
Here, the nonlinear term $\left(1-\Delta\right)^{-1}\Delta\left(\mathcal{m}^4+4\mathcal{m}^3+6\mathcal{m}^2\right)$ can be estimated by using Hausdorff-Young inequality, such as
\begin{align*}
&\left\|\left(1-\Delta\right)^{-1}\Delta\left(\mathcal{m}^4+4\mathcal{m}^3+6\mathcal{m}^2\right)(t)\right\|^2
\lesssim\left\|\frac{|\bm{\xi}|^2}{|\bm{\xi}|^2+1}\left(\widehat{\mathcal{m}^4}+4\widehat{\mathcal{m}^3}+6\widehat{\mathcal{m}^2}\right)(t)\right\|^2\\
&\qquad\quad\ \ \lesssim\left\||\bm{\xi}|^2\left(\widehat{\mathcal{m}^4}+4\widehat{\mathcal{m}^3}+6\widehat{\mathcal{m}^2}\right)(t)\right\|^2
\lesssim\left\|\nabla^2\left(\mathcal{m}^4+4\mathcal{m}^3+6\mathcal{m}^2\right)(t)\right\|^2
\lesssim\|\mathcal{m}(t)\|_2^2.
\end{align*}
Substitute \eqref{11} and \eqref{12} into \eqref{0jie} to get
\begin{align}\notag
\|(\mathcal{n},\bm{u},\mathcal{m})(t)\|
&\leq C\|\bm{U}_0\|_{L^1\cap L^2}(1+t)^{-\frac 34}
+CG_0(t)\int_{0}^{t}(1+t-\tau)^{-\frac 34}(1+\tau)^{-\frac{3}{2}}d\tau\\ \label{num0}
&\lesssim (1+t)^{-\frac 34}\left(\|\bm{U}_0\|_{L^1\cap L^2}+G_0(t)\right),
\end{align}
where we have used that
\begin{align*}
&\int_{0}^{t}(1+t-\tau)^{-\frac 34}(1+\tau)^{-\frac{3}{2}}d\tau\\
=&\int_{0}^{\frac t2}(1+t-\tau)^{-\frac 34}(1+\tau)^{-\frac{3}{2}}d\tau
+\int_{\frac t2}^{t}(1+t-\tau)^{-\frac 34}(1+\tau)^{-\frac{3}{2}}d\tau\\
\leq&(1+\frac t2)^{-\frac 34}\int_{0}^{\infty}(1+\tau)^{-\frac{3}{2}}d\tau
+(1+\frac t2)^{-\frac{3}{2}}\int_{\frac t2}^{t}(1+t-\tau)^{-\frac 34}d\tau\\[3mm]
\lesssim&(1+t)^{-\frac{3}{4}}+(1+t)^{-\frac{5}{4}}-(1+t)^{-\frac{3}{2}}\lesssim(1+t)^{-\frac{3}{4}}.
\end{align*}
From \eqref{basic}, \eqref{apriori1} and \eqref{apriori2}, we have 
\begin{equation*}
\frac{\rm d}{{\rm d}t}\left(\|(\mathcal{n},\bm{u},\mathcal{m})\|^2_2
+\delta\sum_{|\alpha|\leq1}\left\langle\partial_{\bm{x}}^{\alpha}\bm{u},\nabla\partial_{\bm{x}}^{\alpha}\mathcal{n}\right\rangle\right)
+\|\nabla(\mathcal{n},\bm{u},\mathcal{m})\|^2_1
\leq 0,
\end{equation*}
which leads to 
\begin{equation}\label{F1ieq}
\frac{\rm d}{{\rm d}t}\left(\|(\mathcal{n},\bm{u},\mathcal{m})\|^2_2
+\delta\sum_{|\alpha|\leq1}\left\langle\partial_{\bm{x}}^{\alpha}\bm{u},\nabla\partial_{\bm{x}}^{\alpha}\mathcal{n}\right\rangle\right)
+\|(\mathcal{n},\bm{u},\mathcal{m})\|^2_2
\leq\|(\mathcal{n},\bm{u},\mathcal{m})\|^2.
\end{equation}
Then \eqref{F1ieq} can be rewritten as
\begin{equation}\label{dF1}
\frac{\rm d}{{\rm d}t}F_0(t)
+F_0(t)
\leq\|(\mathcal{n},\bm{u},\mathcal{m})\|^2.
\end{equation}
Hence, by using the Gronwall's inequality, \eqref{dF1} and \eqref{num0}, we have that
\begin{align*}
F_0(t)
&\leq F_0(0)e^{-t}+C\int_{0}^{t}e^{-(t-\tau)}\|(\mathcal{n},\bm{u},\mathcal{m})(\tau)\|^2\ d\tau\\
&\lesssim F_0(0)e^{-t}+\int_{0}^{t}e^{-(t-\tau)}(1+\tau)^{-\frac 32}\ d\tau\left(\|(\mathcal{n}_0,\bm{u}_0,\mathcal{m}_0)\|^2_{L^1\cap L^2}+G_0^2(\tau)\right) \\[2mm]
&\lesssim (1+t)^{-\frac 32}\left(F_0(0)+\|(\mathcal{n}_0,\bm{u}_0,\mathcal{m}_0)\|^2_{L^1\cap L^2}+G_0^2(t)\right)\\[3mm]
&\lesssim(1+t)^{-\frac 32}\left(\|(\mathcal{n}_0,\bm{u}_0,\mathcal{m}_0)\|^2_{L^1\cap H^2}+G_0^2(t)\right),
\end{align*}
which indicates that
\begin{align*}\label{G1}
G_0(t)\lesssim \|(\mathcal{n}_0,\bm{u}_0,\mathcal{m}_0)\|^2_{L^1\cap H^2}+G_0^2(t).
\end{align*}
Now assume \eqref{Gk} holds for some positive integer $i\ (0\leq i\leq k-1)$ and from Strauss Lemma \ref{Strauss}, we know that $G_i(t)$ is bounded. Now it is left to prove \eqref{Gk} also holds for $i+1$. Here we let $i\geq3$, since when $1\leq i\leq2$, \eqref{Gk} can be proved by the same procedure below. First, we estimate the nonlinear terms as follows.
\begin{align*}
&\|({\rm R_1},{\rm R_2},{\rm R_3})(t)\|_{L^1}
\lesssim\|(\mathcal{n},\bm{u},\mathcal{m})(t)\|_1\|\nabla(\mathcal{n},\bm{u},\mathcal{m})(t)\|_1
\lesssim(1+t)^{-\frac 34-\frac 54}G_{i+1}(t),\\
&\|\nabla^{i+1}({\rm R_1},{\rm R_2},{\rm R_3})(t)\|\lesssim(1+t)^{-\frac 34-\frac 12-\frac 34-\frac {i+1}2}G_{i+1}(t)\lesssim(1+t)^{-\frac 34-\frac{i+1}2}G_{i+1}(t),\\
&\|\nabla^{i-2}({\rm R_1},{\rm R_2},{\rm R_3})(t)\|\lesssim(1+t)^{-\frac 34-\frac 12-\frac 34-\frac{i-1}2}G_{i+1}(t)\lesssim(1+t)^{-\frac 34-\frac{i+1}2}G_{i+1}(t).
\end{align*}
From Proposition \ref{th3.6} and  Duhamel's principle, we have that
\begin{align}\notag
\|\nabla^{i+1}&(\mathcal{n},\bm{u},\mathcal{m})(t)\|
\leq\ C\left(\|(\mathcal{n}_0,\bm{u}_0,\mathcal{m}_0)\|_{L^1}+\|\nabla^{i+1}(\mathcal{n}_0,\bm{u}_0,\mathcal{m}_0)\|\right)(1+t)^{-\frac 34-\frac {i+1}2}\\[1.5mm]\notag
+&C  \int_{0}^{\frac t2}(1+t-\tau)^{-\frac 34-\frac {i+1}2}\left(\|({\rm R_1},{\rm R_2},{\rm R_3})(\tau)\|_{L^1}+\|\nabla^{i+1}({\rm R_1},{\rm R_2},{\rm R_3})(\tau)\|\right) d\tau\\\notag
+&C  \int_{\frac t2}^{t}(1+t-\tau)^{-\frac 32}\left(\|\nabla^{i-2}({\rm R_1},{\rm R_2},{\rm R_3})(\tau)\|+\|\nabla^{i+1}({\rm R_1},{\rm R_2},{\rm R_3})(\tau)\|\right) d\tau\\\notag
\lesssim\ &(1+t)^{-\frac 34-\frac {i+1}2}\left(\|(\mathcal{n}_0,\bm{u}_0,\mathcal{m}_0)\|_{L^1}+\|\nabla^{i+1}(\mathcal{n}_0,\bm{u}_0,\mathcal{m}_0)\|\right)\\\notag
+&\int_{0}^{\frac t2}(1+t-\tau)^{-\frac 34-\frac {i+1}2}(1+\tau)^{-\frac 34-\frac 54}G_{i+1}(\tau) d\tau
+\int_{\frac t2}^{t}(1+t-\tau)^{-\frac 32}(1+\tau)^{-\frac 34-\frac{i+1}2}G_{i+1}(\tau) d\tau\\
\lesssim\ &(1+t)^{-\frac 34-\frac {i+1}2}\left(\|(\mathcal{n}_0,\bm{u}_0,\mathcal{m}_0)\|_{L^1}+\|\nabla^{i+1}(\mathcal{n}_0,\bm{u}_0,\mathcal{m}_0)\|+G_{i+1}(t)\right).
\end{align}
(When $1\leq i\leq2$, we don't need to divide the nonlinear estimate into two parts.) With the methods that are used in Lemma \ref{kno0} and \ref{kis0}, we can get it for some positive constant $C$ that
\begin{multline*}
\frac{\rm d}{{\rm d}t}\left(\|\nabla^{i+1}(\mathcal{n},\bm{u},\mathcal{m})\|^2_2
+\delta\sum_{i+1\leq|\alpha|\leq i+2}\left\langle\partial_{\bm{x}}^{\alpha}\bm{u},\nabla\partial_{\bm{x}}^{\alpha}\mathcal{n}\right\rangle\right)
+\|\nabla^{i+2}(\mathcal{n},\bm{u},\mathcal{m})\|^2_1\\[3mm]
\leq\|\nabla^{i+1}({\rm R_1},{\rm R_2},{\rm R_3})(t)\|\|\nabla^{i+1}(\mathcal{n},\bm{u},\mathcal{m})(t)\|,
\end{multline*} 
which leads to 
\begin{multline*}
	\frac{\rm d}{{\rm d}t}\left(\|\nabla^{i+1}(\mathcal{n},\bm{u},\mathcal{m})\|^2_2
	+\delta\sum_{i+1\leq|\alpha|\leq i+2}\left\langle\partial_{\bm{x}}^{\alpha}\bm{u},\nabla\partial_{\bm{x}}^{\alpha}\mathcal{n}\right\rangle\right)
	+\|\nabla^{i+1}(\mathcal{n},\bm{u},\mathcal{m})\|^2_2\\[3mm]
	\leq\|\nabla^{i+1}({\rm R_1},{\rm R_2},{\rm R_3})(t)\|\|\nabla^{i+1}(\mathcal{n},\bm{u},\mathcal{m})(t)\|+\|\nabla^{i+1}(\mathcal{n},\bm{u},\mathcal{m})(t)\|^2.
\end{multline*}	
And that is
\begin{equation*}
\frac{\rm d}{{\rm d}t}F_{i+1}(t)
+F_{i+1}(t)
\lesssim(1+t)^{-\frac 32-(i+1)}\left(\|(\mathcal{n}_0,\bm{u}_0,\mathcal{m}_0)\|^2_{L^1}+\|\nabla^{i+1}(\mathcal{n}_0,\bm{u}_0,\mathcal{m}_0)\|^2+G_{i+1}^2(t)\right).
\end{equation*}
Hence, by using the Gronwall's inequality, we have that
\begin{align*}
F_{i+1}(t)
&\lesssim F_{i+1}(0)e^{-t}+\int_{0}^{t}e^{-(t-\tau)}(1+\tau)^{-\frac 32-(i+1)}\left(\|\bm{U}_0\|^2_{L^1}+\|\nabla^{i+1}\bm{U}_0\|^2+G_{i+1}^2(\tau)\right)d\tau\\
&\lesssim (1+t)^{-\frac 32-(i+1)}\left(F_{i+1}(0)+\|\bm{U}_0\|^2_{L^1}+\|\nabla^{i+1}\bm{U}_0\|^2+G_{i+1}^2(t)\right),
\end{align*}
which indicates that
\begin{align}\label{Gi+1}
(1+t)^{\frac 32+(i+1)}F_{i+1}(t)
\lesssim\|\bm{U}_0\|^2_{L^1}+\|\nabla^{i+1}\bm{U}_0\|^2_2+G_{i+1}^2(t).
\end{align}
Now add \eqref{Gk} to \eqref{Gi+1}, we have
\begin{align*}\label{G}
G_{i+1}(t)
\lesssim\|\bm{U}_0\|^2_{L^1\cap H^{i+3}}+G_{i+1}^2(t),
\end{align*}
which completes the induction. With Strauss Lemma \ref{Strauss}, we can see that $G_k(t)$ is bounded, which leads to \eqref{k}. Apply $\partial_{\bm{x}}^{\alpha}\ (0\leq|\alpha|\leq k+1)$ to $\eqref{cns2}_4$, multiply it by $\partial_{\bm{x}}^{\alpha}\bm{q}$ and integrate the resultant equation over $\mathbb{R}^3$ to get
\begin{align*}
	\|\partial_{\bm{x}}^{\alpha}{\rm div}\bm{q}\|^2+\|\partial_{\bm{x}}^{\alpha}\bm{q}\|^2
	\leq 4\|\partial_{\bm{x}}^{\alpha}\nabla\mathcal{m}\|^2+\|\partial_{\bm{x}}^{\alpha}\nabla\left(\mathcal{m}^4\right)\|^2+4\|\partial_{\bm{x}}^{\alpha}\nabla\left(\mathcal{m}^3\right)\|^2+6\|\partial_{\bm{x}}^{\alpha}\nabla\left(\mathcal{m}^2\right)\|^2.
\end{align*}
We can see the fact from \eqref{k} that the nonlinear terms decay faster than the linear term $4\|\partial_{\bm{x}}^{\alpha}\nabla\mathcal{m}\|^2$, which gives \eqref{q}.
\end{proof}
\end{lemma}
\begin{lemma}\label{lemma4.2}
	If there exists a small enough constant $\epsilon_3>0$ such that the initial data satisfies  
	\begin{align*}
	&\|(\mathcal{n}_0,\bm{u}_0,\mathcal{m}_0)\|_{L^1\cap H^{k}}\leq\epsilon_3\quad(\ \text{for any integer}\ k\geq2),\\[1.5mm]
    &C'\|\mathcal{n}_0\|_{L^1}^2-c_0'\|(\bm{u}_0,\mathcal{m}_0)\|_{L^1}^2>C\epsilon_3^2,
	\end{align*}
	where constants $C'$ and $c'$ are determined in Proposition \ref{Pro9}.
	Then, there exist two positive constants $a'_1,a'_2$ such that
	\begin{equation*}
	a'_1(1+t)^{-\frac{3}{4}}\leq \|(\mathcal{n},\bm{u},\mathcal{m},{\rm div}\bm{q})(t)\|\leq a'_2(1+t)^{-\frac{3}{4}},\qquad t\geq0.
	\end{equation*}
\begin{proof}
	
   From Duhamel's principle, it holds that
	\begin{eqnarray*}
		\|(\mathcal{n},\bm{u},\mathcal{m})(t,\bm{x})-e^{t\bm{B}}(\mathcal{n}_0,\bm{u}_0,\mathcal{m}_0)(\bm{x})\|
		&&\leq \int_{0}^{t}\|e^{(t-\tau)\bm{B}}({\rm R_1},{\rm R_2},{\rm R_3})(\tau)\|{\rm d}\tau\\
		&&\lesssim\int_{0}^{t}(1+t-\tau)^{-\frac{3}{4}}\|({\rm R_1},{\rm R_2},{\rm R_3})(\tau)\|_{L^1\cap L^2}\ d\tau
	\end{eqnarray*}
Now with Lemma \eqref{nablak}, the nonlinear source terms can be also estimated as
\begin{align*}
&\|({\rm R_1},{\rm R_2},{\rm R_3})(t)\|_{L^1}
\lesssim\|(\mathcal{n},\bm{u},\mathcal{m})(t)\|_2^2
\lesssim (1+t)^{-\frac 32}\|(\mathcal{n}_0,\bm{u}_0,\mathcal{m}_0)\|^2_{L^1\cap H^{2}},\\
&\|({\rm R_1},{\rm R_2},{\rm R_3})(t)\|\lesssim\|(\mathcal{n},\bm{u},\mathcal{m})(t)\|_2^2\lesssim (1+t)^{-\frac 32}\|(\mathcal{n}_0,\bm{u}_0,\mathcal{m}_0)\|^2_{L^1\cap H^{2}}.
\end{align*}
So we have
\begin{align*}
	\|(\mathcal{n},\bm{u},\mathcal{m})(t,\bm{x})-e^{t\bm{B}}(\mathcal{n}_0,\bm{u}_0,\mathcal{m}_0)(\bm{x})\|
	\leq C\epsilon_3^2\int_{0}^{t}(1+t-\tau)^{-\frac{3}{4}}(1+t)^{-\frac 32}\ d\tau
\end{align*}	
thus, from Proposition \ref{Pro9}, we know
\begin{eqnarray}\label{updown}	
	a_1(1+t)^{-\frac34}-C\epsilon_3^2(1+t)^{-\frac 34}	\leq\|(\mathcal{n},\bm{u},\mathcal{m})(t)\|\leq a_2(1+t)^{-\frac34}+C\epsilon_3^2(1+t)^{-\frac 34}.
\end{eqnarray}	
 As for ${\rm div}\bm{q}$, we can see it from \eqref{divq} that
\begin{align*}
    \|{\rm div}\bm{q}(t)\|^2
    -4\left\|\mathcal{m}(t)\right\|^2
	&\lesssim\|{\rm div}\bm{q}(t)\|^2
	-4\left\|\frac{|\bm{\xi}|^2}{|\bm{\xi}|^2+1}\widehat{\mathcal{m}}(t)\right\|^2
	\lesssim\left\|\frac{|\bm{\xi}|^2}{|\bm{\xi}|^2+1}\left(\widehat{\mathcal{m}^4}+4\widehat{\mathcal{m}^3}+6\widehat{\mathcal{m}^2}\right)(t)\right\|^2\\[2mm]
	&\lesssim\left\|\left(\mathcal{m}^4+4\mathcal{m}^3+6\mathcal{m}^2\right)(t)\right\|^2\lesssim(1+t)^{-3}\|(\mathcal{n}_0,\bm{u}_0,\mathcal{m}_0)\|_{L^1\cap H^{2}},
\end{align*}
which combined with \eqref{updown} provides the estimate for ${\rm div}\bm{q}$. Let $a'_1=\frac{a_1-C\epsilon_3^2}{2}$ and $a'_2=2\left(a_2+C\epsilon_3^2\right)$.
Now the proof is completed.
\end{proof}
\end{lemma}
Finally, it is easy to see that Theorem \ref{th1.1} follows from Proposition \ref{existence} and \ref{a priori}; Theorem \ref{nablak1} comes from Lemma \ref{nablak} and Lemma \ref{lemma4.2} indicates Theorem \ref{updown0}.
\bigbreak

\section{Acknowledgements.} Guiqiong Gong was supported by the Fundamental Research Funds for the Central Universities under contract 2682022CX045. The work was supported by the grants from the National Natural Science Foundation of China under contracts 11731008, 11671309 and 11971359. The authors express much gratitude to Professor Huijiang Zhao for his support and his suggestion.\\

\bibliographystyle{abbrvnat}

\bibliography{references}

\end{document}